\documentclass[a4paper,11pt]{amsart}
\usepackage{amsmath,amsthm,amssymb,amsfonts,enumerate,color}
\usepackage[pdftex]{graphicx}

\newcommand{\R}{\mathbb{R}}

\newcommand{\supp}{\operatorname{supp}}

\newcommand{\tr}{\operatorname{tr}}
\newcommand{\x}{\mathbf{x}}
\newcommand{\y}{\mathbf{y}}
\newcommand{\z}{\mathbf{z}}
\newcommand{\w}{\mathbf{w}}
\newcommand{\X}{\mathbf{X}}

\newtheorem{thm}{Theorem}[section]

\newtheorem{lem}[thm]{Lemma}

\theoremstyle{definition}
\newtheorem{defn}[thm]{Definition}
\newtheorem{rem}[thm]{Remark}

\numberwithin{equation}{section}

\allowdisplaybreaks

\author[P. R. Stinga]{Pablo Ra\'ul Stinga}
    \address{Department of Mathematics\\
    Iowa State University\\
    396 Carver Hall, Ames, IA 50011, USA}
    \email{stinga@iastate.edu}

\author[J. L. Torrea]{Jos\'e L. Torrea}
    \address{Departamento de Matem\'aticas\\
    Universidad Aut\'onoma de Madrid\\
    28049 Madrid, Spain}
\email{joseluis.torrea@uam.es}



\thanks{Research supported by Simons Foundation grant 580911 (PRS) and by
grant MTM2015-66157-C2-1-P (MINECO/FEDER) from Government of Spain.}

\begin{document}

\title[Mixed-norm estimates]{H\"older, Sobolev, weak-type and $BMO$ 
estimates \\ in mixed-norm with weights for parabolic equations}

\begin{abstract}
We prove weighted mixed-norm $L^q_t(W^{2,p}_x)$ and $L^q_t(C^{2,\alpha}_x)$
estimates for $1<p,q<\infty$ and $0<\alpha<1$,
weighted mixed weak-type estimates for $q=1$, $L^\infty_{t}(L^p_x)-BMO_t(W^{2,p}_x)$,
and $L^\infty_{t}(C^\alpha_x)-BMO_t(C^{2,\alpha}_x)$, and a.e.~pointwise formulas for derivatives,
for solutions $u=u(t,x)$ to parabolic equations of the form
$$\partial_tu-a^{ij}(t)\partial_{ij}u+u=f\quad t\in\R,~x\in\R^n$$
and for the Cauchy problem
$$\begin{cases}
\partial_tv-a^{ij}(t)\partial_{ij}v+v=f&\hbox{for}~t >0,~x\in\R^n \\
v(0,x)=g&\hbox{for}~x\in\R^n.
\end{cases}$$
The coefficients $a(t)=(a^{ij}(t))$ are just bounded, measurable, symmetric
and uniformly elliptic. Furthermore, we show strong, weak type and $BMO$-Sobolev estimates with
parabolic Muckenhoupt weights. It is quite remarkable that most of our results are new even for the classical heat equation
$$\partial_tu-\Delta u+u=f.$$
\end{abstract}

\maketitle

\section{Introduction}

In this paper we prove weighted mixed-norm $L^q_t(W^{2,p}_x)$ and $L^q_t(C^{2,\alpha}_x)$
estimates for $1<p,q<\infty$ and $0<\alpha<1$,
weighted mixed weak-type estimates for $q=1$, $1<p<\infty$ and $0<\alpha<1$, and weighted $L^\infty_{t}(L^p_x)-BMO_t(W^{2,p}_x)$
and $L^\infty_{t}(C^\alpha_x)-BMO_t(C^{2,\alpha}_x)$ estimates for solutions $u=u(t,x)$ to the parabolic equation
\begin{equation}\label{EQ1}
\partial_tu-a^{ij}(t)\partial_{ij}u+u=f\quad t\in\R,~x\in\R^n
\end{equation}
and solutions $v=v(t,x)$ of the Cauchy problem
\begin{equation}\label{EQ2}
\begin{cases}
\partial_tv-a^{ij}(t)\partial_{ij}v+v=f&\hbox{for}~t >0,~x\in\R^n \\
v(0,x)=g&\hbox{for}~x\in\R^n.
\end{cases}
\end{equation}
The matrix of coefficients $a(t)=(a^{ij}(t))$ is assumed to be just bounded, measurable, symmetric $a^{ij}(t)=a^{ji}(t)$
and uniformly elliptic, that is, there exists $\Lambda>0$ such that
$\Lambda|\xi|^2\leq a^{ij}(t)\xi_i\xi_j\leq\Lambda^{-1}|\xi|^2$,
for all $\xi\in\R^n$, for a.e.~$t\in\R$. We also show strong, weak type and $BMO$-Sobolev estimates with
parabolic weights. Furthermore, we present explicit pointwise formulas
for the derivatives of the solutions.

Our first result is concerned with classical solvability and pointwise formulas
for derivatives of solutions $u$ to \eqref{EQ1}. We define the matrices
$$A_{t,\tau}=(A^{ij}(t,\tau)) =\int_{t-\tau}^t a^{ij}(r)\,dr\quad\hbox{for}~t\in\R,~\tau>0.$$
Then $A_{t,\tau}$ verifies
$\Lambda\tau|\xi|^2\leq A^{ij}(t,\tau)\xi_i\xi_j\leq\Lambda^{-1}\tau|\xi|^2$, for all $t\in\R$, $\tau>0$.
Let
$$B_{t,\tau}=(B^{ij}(t,\tau))=(A_{t,\tau})^{-1}$$
the inverse of $A_{t,\tau}$. Consider the following kernel
\begin{equation}\label{GW1}
p(t,\tau,x):= \chi_{\tau>0}e^{-\tau}\frac{\exp\big(-\tfrac{1}{4}\langle B_{t,\tau}x,x\rangle\big)}{(4\pi)^{n/2} (\det B_{t,\tau})^{-1/2}}
\quad\hbox{for}~t,\tau\in\R,~x\in\R^n.
\end{equation}

\begin{thm}[Classical solvability]\label{fundamental 1}
Let $f=f(t,x)\in L^p(\R^{n+1})$, $1\leq p\leq\infty$, and define
\begin{equation}\label{u}
u(t,x)=\int_{\mathbb{R}^{n+1}}p(t,\tau,y)f(t-\tau,x-y)\,dy\,d\tau.
\end{equation}
Then
$$\|u\|_{L^p(\R^{n+1})}\leq\|f\|_{L^p(\R^{n+1})}.$$
If $f\in C_c^2(\R^{n+1})$ then $u$ is the unique bounded classical solution to \eqref{EQ1}
and the following pointwise limit formulas for its derivatives hold:
\begin{equation}\label{equation2 x}
\partial_{ij}u(t,x)=\lim_{\varepsilon\rightarrow0^+}\int_{\Omega_\varepsilon}
\partial_{y_iy_j}p(t,\tau,y)f(t-\tau,x-y)\,dy\,d\tau-f(t,x)I_{ij}(a)(t)
\end{equation}
and
\begin{equation}\label{equation2 t}
\partial_t u(t,x) =\lim_{\varepsilon\rightarrow0^+}\int_{\Omega_\varepsilon}
(\partial_t+ \partial_\tau) p(t,\tau,y)f(t-\tau,x-y)\,dy\,d\tau+f(t,x)J(a)(t).
\end{equation}
Here $\Omega_\varepsilon= \{(\tau,y)\in\R^{n+1}:\max(|\tau|^{1/2},|y|)>\varepsilon\}$
and, for $i,j=1,\ldots,n$ and $t\in\R$,
\begin{equation}\label{matrizA}
I_{ij}(a)(t)=\int_{\{x: |2 a(t)^{1/2}x|\ge 1\}}  
\frac{e^{-|x|^2}}{\pi^{n/2}}\frac{(a(t)^{-1/2}x)_j (a(t)^{1/2} x)_i}{ |a(t)^{1/2} x|^2}\,dx
\end{equation}
and
\begin{equation}\label{matrizb}
J(a)(t) = \int_{\{x: |2a(t)^{1/2}x|\le 1\}}\frac{e^{-|x|^2}}{\pi^{n/2}}\,dx.
\end{equation}
\end{thm}

It is important to stress that, up to our knowledge, this is the first time the terms
\eqref{matrizA} and \eqref{matrizb} are explicitly computed in the formulas for the derivatives \eqref{equation2 x}
and \eqref{equation2 t} for solutions $u$ to a time-dependent coefficients equation like \eqref{EQ1}.
As our next results will show, such representations also hold a.e. and in the corresponding norms in the 
case when $f$ belongs to more general functional spaces.

The natural geometric setting for parabolic equations is the parabolic metric space,
see Remark \ref{rem:parabolic metric space}.
The class of parabolic Muckenhoupt weights $PA_p(\R^{n+1})$
is the suited one for weighted Sobolev estimates.
Observe that $PA_p(\R^{n+1})$ is different from the usual Muckenhoupt $A_p(\R^{n+1})$ class. In our next result we
show weighted parabolic Sobolev estimates for \eqref{EQ1}.

\begin{thm}[Weighted parabolic Sobolev estimates]\label{thm:parabolic weights}
Let $f\in L^p(\mathbb{R}^{n+1},w)$, for some $1\le p<\infty$ and $w\in PA_p(\R^{n+1})$.
Then $u$ defined as in \eqref{u} is in $L^p(\R^{n+1},w)$, with
\begin{equation}\label{eq:uweights}
\|u\|_{L^p(\R^{n+1},w)}\leq C_{n,p,\Lambda,w}\|f\|_{L^p(\R^{n+1},w)}.
\end{equation}
Moreover, the following estimates hold.
\begin{enumerate}[$(1)$]
\item If $1<p<\infty$ then $\partial_{ij}u,\partial_tu\in L^p(\R^{n+1},w)$ and
$$\|\partial_{ij}u\|_{L^p(\R^{n+1},w)}+\|\partial_t u\|_{L^p(\R^{n+1},w)}\le
C_{n,p,\Lambda,w}\|f\|_{L^p(\R^{n+1},w)}.$$
\item If $p=1$ then $\partial_{ij}u,\partial_tu\in\,$weak-$L^1(\R^{n+1},w)$ and, for any $\lambda>0$,
$$w\big(\{(t,x)\in\R^{n+1}:|\partial_{ij}u|+|\partial_tu|>\lambda\}\big)
\le\frac{C_{n,\Lambda,w}}{\lambda}\|f\|_{L^1(\R^{n+1},w)}.$$
\end{enumerate}
In both cases the representation formulas \eqref{equation2 x} and \eqref{equation2 t}
hold true as limits in $L^p(\R^{n+1},w)$ when $1<p<\infty$, in the measure $w(t,x)dtdx$ when $p=1$,
and for a.e. $(t,x)\in\R^{n+1}$.
\end{thm}

The case $p=\infty$ was left out of Theorem \ref{thm:parabolic weights}.
We address this next, where we obtain a weighted parabolic Sobolev estimate
for the sharp maximal function.
For the notation see Section \ref{Section:eth}, in particular,
\eqref{eq:sharp maximal} for the definition of $M^\#_F$, where $F$ is a Banach space.

\begin{thm}[Weighted parabolic $BMO$ estimate]\label{thm:parabolic-BMO}
Let $u$ be as in \eqref{u}, where $f\in L_c^\infty(\R^{n+1})$. Let
$w=w(t,x)>0$ such that $w^{-1}\in PA_1(\R^{n+1})$ and suppose that $wf\in L^\infty(\R^{n+1})$.
Then the following $BMO$ estimate with weights for $\partial_{ij}u$ and $\partial_tu$ holds:
$$\|w\cdot M_{\R}^\#(\partial_{ij}u)\|_{L^\infty(\R^{n+1})}+
\|w\cdot M^\#_{\R}(\partial_tu)\|_{L^\infty(\R^{n+1})}\leq C_{n,\Lambda,w}\|wf\|_{L^\infty(\R^{n+1})}.$$
\end{thm}

We now present our mixed-norm $L^q(\R,\nu;W^{2,p}(\R^n,\omega))$ estimates,
where $\nu$ and $\omega$ are Muckenhoupt $A_q(\R)$ and $A_p(\R^n)$ weights, respectively.

\begin{thm}[Mixed-norm weighted $L^q$--Sobolev estimates]\label{thm:LqSobolev}
Let $f\in L^q(\R,\nu;L^p(\R^{n},\omega))$ for some $1\leq p,q<\infty$, where
$\nu\in A_q(\R)$ and $\omega \in A_p(\R^n)$. Then $u$ defined as in \eqref{u} is in $L^q(\R,\nu;L^p(\R^{n},\omega))$, with
\begin{equation}\label{eq:umixedweights}
\|u\|_{L^q(\R,\nu;L^p(\R^{n},\omega))}\leq C_{n,p,q,\Lambda,\nu,\omega}\|f\|_{L^q(\R,\nu;L^p(\R^{n},\omega))}.
\end{equation}
Moreover, the following estimates hold.
\begin{enumerate}[$(i)$]
\item If $1<p,q<\infty$ then $\partial_{ij}u,\partial_tu\in L^q(\R,\nu;L^{p}(\R^n,\omega))$ and
$$\|\partial_{ij}u\|_{L^q(\R,\nu;L^p(\R^{n},\omega))}+\|\partial_tu\|_{L^q(\R,\nu;L^p(\R^{n},\omega))}
\le C_{n,p,q,\Lambda,\nu,\omega}\|f\|_{L^q(\R,\nu;L^p(\R^{n},\omega))}.$$
In this case, the representation formulas \eqref{equation2 x} and \eqref{equation2 t}
hold true as limits in the norm of $L^q(\R,\nu;L^p(\R^{n},\omega))$ and for a.e.~$(t,x)\in\R^{n+1}$.
\item If $q=1$ and $1 < p<\infty$ then $\partial_{ij}u,\partial_tu\in\,$weak-$L^1(\R,\nu;L^p(\R^n,\omega))$
and, for any $\lambda>0$,
$$\nu\big( \{t \in\R:\|\partial_{ij}u(t,\cdot)\|_{L^p(\R^{n},\omega)}+\|\partial_{t}u(t,\cdot)\|_{L^p(\R^{n},\omega)}
>\lambda\}\big)\le\frac{C_{n,p,\Lambda,\nu,\omega}}{\lambda}\|f\|_{L^1(\R,\nu;L^p(\R^{n},\omega))}.$$
In this case, the representation formulas \eqref{equation2 x} and \eqref{equation2 t}
hold true as limits in the measure $\nu(t)dt$ and in the
norm of $L^p(\R^n,\omega)$, and for a.e.~$(t,x)\in\R^{n+1}$.
\end{enumerate}
\end{thm}

Notice that Theorem \ref{thm:LqSobolev} is neither a particular case
of Theorem \ref{thm:parabolic weights} nor implies it. For the endpoint case $q=\infty$ we have the following
estimates. Recall the notation in \eqref{eq:sharp maximal}.

\begin{thm}[Mixed-norm weighted $L^\infty$-$BMO$ Sobolev estimate]\label{thm:mixed-BMO}
Let $u$ be as in \eqref{u}, where $f\in L_c^\infty(\R;L^p(\R^n,\omega))$, for some $1<p<\infty$
and $\omega\in A_p(\R^n)$. Let $\nu=\nu(t)>0$ such that $\nu^{-1}\in A_1(\R)$ and suppose that
$\nu f\in L^\infty(\R;L^p(\R^n,\omega))$. Then the following $L^p$-$BMO$ mixed-norm
weighted estimate for $\partial_{ij}u$ and $\partial_tu$ holds:
$$\|\nu\cdot M_{L^p(\R^n,\omega)}^\#(\partial_{ij}u)\|_{L^\infty(\R)}
+\|\nu\cdot M_{L^p(\R^n,\omega)}^\#(\partial_tu)\|_{L^\infty(\R)}
\leq C_{n,p,\Lambda,\omega,\nu}\|\nu f\|_{L^\infty(\R;L^p(\R^n,\omega))}.$$
\end{thm}

Our third set of main results regards mixed-norm $L^q(\R,\nu;C^{2,\alpha}(\R^n))$ estimates,
where $\nu\in A_q(\R)$ if $1\leq q<\infty$, $\nu\equiv1$ if $q=\infty$, and $0<\alpha<1$.

\begin{thm}[Mixed-norm $L^q$--H\"older estimates]\label{thm:Holder}
Let $f\in L^q(\R,\nu;C^{\alpha}(\R^{n}))$ for some $1\leq q\leq\infty$,
where $\nu\in A_q(\R)$ if $1\leq q<\infty$, $\nu\equiv1$ if $q=\infty$, and $0<\alpha<1$.
Then $u$ defined as in \eqref{u} is in $L^q(\R,\nu;C^{\alpha}(\R^{n}))$, with
$$\|u\|_{L^q(\R,\nu;C^{\alpha}(\R^{n}))}\leq C_{n,q,\alpha,\Lambda,\nu}\|f\|_{L^q(\R,\nu;C^{\alpha}(\R^{n}))}.$$
Moreover, the following estimates hold.
\begin{enumerate}[$(i)$]
\item If $1<q\leq\infty$ then $\partial_{ij}u,\partial_tu\in L^q(\R,\nu;C^{\alpha}(\R^n))$ and
$$\|\partial_{ij}u\|_{L^q(\R,\nu;C^{\alpha}(\R^{n}))}+\|\partial_tu\|_{L^q(\R,\nu;C^{\alpha}(\R^{n}))}
\le C_{n,q,\alpha,\Lambda,\nu}\|f\|_{L^q(\R,\nu;C^{\alpha}(\R^{n}))}.$$
In this case, the representation formulas \eqref{equation2 x} and \eqref{equation2 t}
hold true as limits in the norm of $L^q(\R,\nu;C^{\alpha}(\R^{n}))$, and for a.e.~$t\in\R$ and
uniformly in $x\in\R^{n}$.
\item If $q=1$ then $\partial_{ij}u,\partial_tu\in\,$weak-$L^1(\R,\nu;C^{\alpha}(\R^n))$
and, for any $\lambda>0$,
$$\nu\big( \{t \in\R:[\partial_{ij}u(t,\cdot)]_{C^\alpha(\R^{n})}+[\partial_{t}u(t,\cdot)]_{C^\alpha(\R^{n})}>\lambda\}\big)
\le\frac{C_{n,\alpha,\Lambda,\nu}}{\lambda}\|f\|_{L^1(\R,\nu;C^{\alpha}(\R^{n}))}.$$
In this case, the representation formulas \eqref{equation2 x} and \eqref{equation2 t}
hold true as limits in the measure $\nu(t)dt$ and in 
the norm in $C^{\alpha}(\R^n)$, and for a.e.~$t\in\R$ and uniformly in $x\in\R^{n}$.
\end{enumerate}
\end{thm}

We next turn our attention to the Cauchy problem \eqref{EQ2}.
The next statement proves the existence and uniqueness of a classical solution
and the representation formulas for its derivatives. We denote $\R^{n+1}_+=(0,\infty)\times\R^n$.

\begin{thm}[Classical solvability for the Cauchy problem]\label{fundamental 3}
Let $g=g(x)\in L^p(\R^n)$ and $f=f(t,x)\in L^p(\R^{n+1}_+)$, $1\leq p\leq\infty$, and define
\begin{equation}\label{eq:v}  
v(t,x)=\int_0^t\int_{\R^n}p(t,\tau,y)f(t-\tau,x-y)\,dy\,d\tau+\int_{\R^n}p(t,t,y) g(x-y)\,dy.
\end{equation}
Then
$$\|v\|_{L^p(\R^{n+1})}\leq\|f\|_{L^p(\R^{n+1}_+)}+\|g\|_{L^p(\R^n)}.$$
If $f\in C^2_c(\R_+^{n+1})$ and $g\in C^2_c(\R^n)$  then $v$ is the unique bounded
classical solution to the Cauchy problem
$$\begin{cases}
\partial_tv-a^{ij}(t)\partial_{ij}v+v=f&\hbox{for}~t >0,~x\in\R^n \\
v(0,x)=g&\hbox{for}~x\in\R^n.
\end{cases}$$
In this case the following pointwise limit formulas hold:
\begin{equation}\label{equation3 x}
\begin{aligned}
\partial_{ij}v(t,x)&=\lim_{\varepsilon\rightarrow 0}\int_\varepsilon^t\int_{\mathbb{R}^n}
\partial_{y_iy_j}p(t,\tau,x-y)f(t-\tau,y)\,dy\,d\tau \\
&\quad+\int_{\mathbb{R}^n}\partial_{y_iy_j}p(t,t,y)g(x-y)\,dy
\end{aligned}
\end{equation}
and
\begin{equation}\label{equation3 t}
\begin{aligned}
\partial_{t}v(t,x) &=\lim_{\varepsilon \rightarrow 0}\int_\varepsilon^t \int_{\mathbb{R}^n}(\partial_t+ \partial_{\tau})p(t,\tau,x-y)f(t-\tau,y) \,dy\, d\tau \\
&\quad+\int_{\mathbb{R}^n}\partial_tp(t,t,y) g(x-y)\,dy+f(t,x).
\end{aligned}
\end{equation}
\end{thm}

Our last main result contains the mixed-norm estimates
and the formulas for derivatives for the Cauchy problem \eqref{EQ2} when $g=0$.

\begin{thm}\label{thm:Cauchy}
Suppose that $f$ satisfies the assumptions in any of Theorems
\ref{thm:parabolic weights}, \ref{thm:parabolic-BMO}, \ref{thm:LqSobolev}, \ref{thm:mixed-BMO}
or \ref{thm:Holder}, with $\R^{n+1}_+$ in place of $\R^{n+1}$.
Let $v$ be the solution to the Cauchy problem
$$\begin{cases}
\partial_tv-a^{ij}(t)\partial_{ij}v+v=f&\hbox{for}~t >0,~x\in\R^n \\
v(0,x)=0&\hbox{for}~x\in\R^n
\end{cases}$$
given by \eqref{eq:v}. Then $v$ satisfies the corresponding estimates in those Theorems, with $\R^{n+1}_+$
in place of $\R^{n+1}$. Moreover, formulas \eqref{equation3 x} and \eqref{equation3 t}
for the derivatives of $v$ hold true in the norm and a.e./pointwise senses as
stated there, with $\R^{n+1}_+$ in place of $\R^{n+1}$.
\end{thm}

Our results were inspired by the fundamental work on mixed-norm $L^q(L^p)$ and
$L^q(C^{\alpha})$ \textit{a priori} estimates for \eqref{EQ1}
by N. V. Krylov \cite{Krylov book, Krylov0, Krylov}.
See also \cite{DKadv, DKtran, PST} for estimates in weighted Sobolev spaces.
Parabolic singular integrals had been considered in the 1960's by Fabes, Sadosky and Jones
\cite{Fabes, Fabes-Sadosky, Jones}.
One of the main tools in Krylov's papers \cite{Krylov0} and \cite{Krylov} is the use of the
vector-valued Calder\'on--Zygmund theory for parabolic singular
integrals.

In this paper we follow the philosophy introduced by A. P. Calder\'on \cite{CalderonBull}
in the elliptic case.

First, we solve \eqref{EQ1}
for compactly supported $C^2$ right hand sides $f$, and obtain the explicit formulas for the
derivatives of $u$ (Theorem \ref{fundamental 1}). 
As it can be seen, $\partial_tu$ and $\partial_{ij}u$ are expressed as principal value singular integrals
acting on $f$ plus a multiplication operator by the bounded functions \eqref{matrizA} and \eqref{matrizb},
see \eqref{equation2 x} and \eqref{equation2 t}. As mentioned before, this seems to be the first time
these terms are explicitly computed. The derivation of the multiplication operators
\eqref{matrizA} and \eqref{matrizb} involves quite delicate real-variable arguments,
see Subsection \ref{derivation}.

In a second step, and relying on the vector-valued version
of Calder\'on's method that was presented in \cite{RRT},
we are able to solve \eqref{EQ1} when the right hand side $f$ is in more general classes,
namely, the weighted $L^q$ spaces (Theorem \ref{thm:parabolic weights}) and
the weighted mixed-norm classes $L^q(L^p)$ and $L^q(C^{\alpha})$ (Theorems \ref{thm:LqSobolev}
and \ref{thm:Holder}, respectively).
More concretely, we shall work with the vector-valued Calder\'on--Zygmund
singular integrals theory in spaces of homogeneous type.
Such machinery requires two ingredients:
a kernel satisfying appropriate size and smoothness estimates, and the
boundedness of the given operator in an $L^{q_0}$ space, for some $1\leq q_0\leq\infty$,
see \cite{Francisco, RT}.
In Theorems \ref{thm:parabolic weights} and \ref{thm:LqSobolev},
which correspond to \textit{parabolic Riesz transforms}, the natural exponent is $q_0=2$.
This is consistent with the usual theory of Riesz transforms for elliptic PDEs considered
by Calder\'on \cite{CalderonBull}, where the Fourier
transform readily shows the $L^2$ continuity.
But for the weighted mixed-norm H\"older estimates of Theorem \ref{thm:Holder},
the initial exponent is $q_0=\infty$. The estimate can be found in Lemma \ref{Lemma6.1}.

On top of all this, with our approach we cover the end point cases $q=1$, where
weak-type estimates are found, and $BMO$ (Theorems \ref{thm:parabolic-BMO}
and \ref{thm:mixed-BMO}). These are also novel, even for the classical heat equation.

To prove the $BMO$
estimates we will need to extend a result on weighted $BMO$ boundedness of singular integrals
from \cite{Segovia-Torrea} to the case of spaces of homogeneous type,
see Theorem \ref{thm:Segovia-Torrea}, which is of independent interest.

The estimates for the Cauchy problem \eqref{EQ2} in Theorem \ref{thm:Cauchy}
are obtained through a delicate comparison argument with the solution $u$ to \eqref{EQ1}.
Indeed, this idea permits us to transfer the already known results for $u$ to $v$,
see Section \ref{section:Cauchy} for details.

\section{Classical solvability: proofs of Theorems \ref{fundamental 1}
and \ref{fundamental 3}}\label{pre}

In this section we present the proofs of Theorems \ref{fundamental 1} and \ref{fundamental 3}.
Towards this end we need a series of preliminary computational lemmas.

\subsection{Computational lemmas}

This subsection is devoted to several derivation formulas and estimates that will be useful
for the proofs of Theorems \ref{fundamental 1} and \ref{fundamental 3}.

\begin{lem}\label{derivativex}
Let $B$ be an $n\times n$ symmetric constant matrix. For any $x=(x_1,\ldots,x_n)\in\R^n$,
\begin{itemize}
\item[(1)] $\displaystyle \partial_{x_ix_j}\exp\big(-\tfrac{1}{4}\langle Bx,x\rangle\big)=
\tfrac{1}{2}\exp\big(-\tfrac{1}{4}\langle Bx,x\rangle\big)\Big[-B_{ij} +\tfrac{1}{2}(Bx)_i(Bx)_j\Big].$
\item[(2)] $\displaystyle \partial_{x_ix_jx_k} \exp\big(-\tfrac{1}{4}\langle Bx,x\rangle\big)\newline
~=  \tfrac{1}{4} \exp\big(-\tfrac{1}{4}\langle Bx,x\rangle\big)
\Big[B_{ij} (Bx)_k +  B_{jk} (Bx)_i+B_{ki} (Bx)_j-\tfrac{1}{2}(B x)_i (B x)_j(Bx)_k \Big].$
\end{itemize}
\end{lem}

\begin{proof}
Denote $x= (x_1,\dots,x_n)$ and $z= (z_1,\dots, z_n)$.
Given $z:=\tfrac{1}{2}Bx$ and $w:=\tfrac{1}{2}x$, consider the function
$f(z,w)=\exp(-\langle z,w\rangle)$. Then
\begin{equation}\label{firstd}
\begin{aligned}
\partial_{x_i} f(z,w) &=
\sum_{k=1}^n\frac{\partial f}{\partial z_k} \frac{\partial z_k}{\partial x_i}+
\sum_{k=1}^n\frac{\partial f}{\partial w_k}\frac {\partial w_k}{\partial x_i} \\
&=-\tfrac{1}{2}\exp(-\langle z,w\rangle)\bigg[\sum_{k=1}^nw_kB_{ki}+ \sum_{k=1}^nz_k\delta_{ki}\bigg] \\
&=-\tfrac{1}{2}\exp(-\langle z, w\rangle)\bigg[\tfrac{1}{2}\sum_{k=1}^nx_k B_{ki}+\tfrac{1}{2}(Bx)_i\bigg] \\
&=-\tfrac{1}{2}\exp(-\langle z, w\rangle)(Bx)_i.
\end{aligned}
\end{equation}
Analogously, we define $g_i(z,w)=-z_i\exp(-\langle z, w\rangle)$. Then 
\begin{align*}
\partial_{x_j} g_i(z,w) &= \tfrac{1}{2}\bigg[\sum_{\ell=1}^n\frac{\partial g_i}{\partial z_\ell} B_{\ell j}
+ \sum_{\ell=1}^n\frac{\partial g_i}{\partial w_\ell}\delta_{\ell j}\bigg] \\
&=  \tfrac{1}{2}\exp(-\langle z, w\rangle)\bigg[ \sum_{\ell=1}^n
(-\delta_{\ell i}+z_iw_\ell)B_{\ell j} + \sum_{\ell=1}^nz_i z_\ell \delta_{\ell j}\bigg] \\
&= \tfrac{1}{2}\exp(-\langle z,w\rangle) \Big[-B_{ij}+ \tfrac{1}{2}(B x)_i (B x)_j \Big].
\end{align*}
Following the same ideas, if we define $h_{ij}(z,w) = \exp(-\langle z,w\rangle)\Big[-\tfrac{1}{2}B_{ij}+ z_iz_j \Big]$, then 
\begin{align*}
\partial_{x_k}& h_{ij}(z,w) \\
&=  \exp(-\langle z, w\rangle)\Big\{-\tfrac{1}{2} (Bx)_k\Big[-\tfrac{1}{2}B_{ij}+ \tfrac{1}{4}(B x)_i (B x)_j \Big] +\Big[
\tfrac{1}{4} B_{ik}(Bx)_j+\tfrac{1}{4} B_{jk}(B x)_i  \Big]\Big\} \\
&=  \exp(-\langle z, w\rangle) \Big\{ -\tfrac{1}{8}(B x)_i (B x)_j(Bx)_k + 
\tfrac{1}{4}\big(B_{ij} (Bx)_k +  B_{jk} (Bx)_i+B_{ki} (Bx)_j\big)\Big\}.
\end{align*}
\end{proof}

\begin{lem}\label{derivativec}
Let $A(t)$ and $B(t)$, $t\in\R$, be nondegenerate, time dependent, $n \times n$ symmetric matrices
with differentiable entries such that $B(t)=A(t)^{-1}$. If we denote $'=\partial_t$ then
\begin{enumerate}[$(i)$]
\item $B' =-BA'B$;
\item $(\det B)'=-(\det B)\tr(A'B)$;
\item $\partial_t\big((\det B)^{\frac{1}{2}}\exp(-\frac{1}{4}\langle Bx, x\rangle)\big)=
(\det B)^{\frac{1}{2}}\exp(-\frac{1}{4}\langle Bx, x\rangle)\Big[-\frac12\tr(A' B)+\frac14\langle A' B x, Bx\rangle\Big]$.
\end{enumerate}
\end{lem}

\begin{proof}
For $(i)$ we just need to observe that $0=I'=(BA)'=B'A+BA'$, so that $B' A =- B A'$.
For $(ii)$ we recall the well known Jacobi's formula
$(\det B)' =\tr((\operatorname{adj}B)B')$, where $\operatorname{adj}B=(\det B)B^{-1}=(\det B)A$.
Hence, by $(i)$,
$$ (\det B)' = -\tr((\det B)A(BA'B))=-(\det B)\tr(A'B).$$
For $(iii)$ we have 
\begin{align*}
\partial_t&\big((\det B)^{1/2}\exp(-\tfrac{1}{4}\langle Bx, x\rangle)\big)= \\
&=\tfrac{1}{2}(\det B)^{-1/2}(\det B)'\exp(-\tfrac{1}{4}\langle Bx, x\rangle)+
(\det B)^{1/2}(-\tfrac{1}{4}\langle B'x, x\rangle)\exp(-\tfrac{1}{4}\langle Bx, x\rangle) \\
&=(\det B)^{1/2}\exp(-\tfrac{1}{4}\langle Bx, x\rangle)\Big[-\tfrac{1}{2}\tr(A'B)+\tfrac{1}{4}\langle A'Bx,Bx\rangle\Big].
\end{align*}
\end{proof}

\begin{lem}\label{nucleo}
The function $p(t,\tau,x)\geq0$ defined in \eqref{GW1} satisfies the following properties.
\begin{enumerate}[$(i)$]
\item For every $t,\tau\in\R$ we have $\displaystyle \int_{\R^n} p(t,\tau,x)\,dx=e^{-\tau}$.
\item For every $t,\tau\in\R$ and $x\in\R^n$,
$$\partial_\tau p(t,\tau,x)=-(\partial_t-a^{ij}(t)\partial_{ij}+1)p(t,\tau,x).$$
\item There exist constants $C,c>0$ depending on $n$ and $\Lambda$ such that, for every $t,\tau\in\R$ and $x\in\R^n$,
\begin{align*}
&0\leq p(t,\tau,x)\le C\chi_{\tau>0}\,e^{-\tau}\,\frac{e^{-|x|^2/(c\tau)}}{\tau^{n/2}}, \\
&|\partial_ip(t,\tau,x)|\le C\chi_{\tau>0}\,e^{-\tau}\,\frac{|x|e^{-|x|^2/(c\tau)}}{\tau^{n/2+1}}, \\
&|\partial_tp(t,\tau,x)|+|\partial_{\tau}p(t,\tau,x)|+|\partial_{ij}p(t,\tau,x)|\le C\chi_{\tau>0}\,e^{-\tau/2}\,\frac{e^{-|x|^2/(c\tau)}}{\tau^{n/2+1}},\\
&|\partial_t\partial_{ij}p(t,\tau,x)|+|\partial_{\tau}\partial_{ij}p(t,\tau,x)|+|\partial_{ijk}p(t,\tau,x)|\le C\chi_{\tau>0}\,e^{-\tau/2}\,\frac{e^{-|x|^2/(c\tau)}}{\tau^{n/2+3/2}}.
\end{align*}
\item The Fourier transform of the function $x \rightarrow p(t,\tau,x)$ is given, for any $\xi\in\R^n$, by
$$\chi_{\tau>0}\frac{e^{-\tau}}{(4\pi)^{n/2}}\exp\big(-|A_{t,\tau}^{1/2}\xi|^2\big)
=\chi_{\tau>0}\frac{e^{-\tau}}{(4\pi)^{n/2}}\exp\big(-\langle A_{t,\tau}\xi,\xi\rangle\big).$$
\item For a function $f\in L^2(\R^{n+1})$ let us define
$$T_\tau f(t,x):=\int_{\R^{n}}p(t,\tau,y)f(t-\tau,x-y)\,dy$$
for $\tau\geq0$ and $(t,x)\in\R^{n+1}$. Then, for any $\tau_1,\tau_2>0$,
$$T_{\tau_1} T_{\tau_2}f(t,x)=T_{\tau_1+\tau_2}f(t,x).$$
\end{enumerate}
\end{lem}

\begin{proof}
In order to check $(i)$ it is enough to perform the change of variables $\bar{x}=B^{1/2}_{t,\tau}x$.
To see $(ii)$ we shall compute the derivatives of the function $p(t,\tau,x)$. As
$$\partial_t A_{t,\tau}=a(t)-a(t-\tau),\quad\partial_\tau A_{t,\tau}=a(t-\tau),$$
by Lemma \ref{derivativec}, 
$$\partial_t p(t,\tau,x)=p(t,\tau,x)\bigg[-\frac12\tr\big((a(t)-a(t-\tau))B_{t,\tau}\big)
+\frac14\langle(a(t)-a(t-\tau))B_{t,\tau}x,B_{t,\tau}x\rangle\bigg],$$
and
$$\partial_\tau p(t,\tau,x)=p(t,\tau,x)\bigg[-\frac12\tr(a(t-\tau)B_{t,\tau})+
\frac14\langle a(t-\tau)B_{t,\tau}x, B_{t,\tau}x\rangle-1\bigg].$$
On the other hand, by Lemma \ref{derivativex},
\begin{align*}
-a^{ij}(t)\partial_{ij}p(t,\tau,x) &= p(t,\tau,x)\bigg[\frac12a^{ij}(t)(B_{t,\tau})_{ij}-
\frac14a^{ij}(t)(B_{t,\tau}x)_i(B_{t,\tau}x)_j\bigg] \\
&= p(t,\tau,x)\bigg[\frac12\tr(a(t)B_{t,\tau})-\frac14\langle a(t)B_{t,\tau}x,B_{t,\tau}x)\rangle\bigg]
\end{align*}
and we get $(ii)$.
Property $(iii)$ follows easily from Lemmas \ref{derivativex} and \ref{derivativec}, the ellipticity of the matrix $a^{ij}(t)$ and by using the equation.
To see $(iv)$ we perform the change of variables $\bar{x}=B^{1/2}_{t,\tau}x$ 
and use the formula for the Fourier transform of an exponential function.
For $(v)$, observe that  the matrix $A^{ij}(t,\tau)$ satisfies
$$A^{ij}(t,\tau_1+\tau_2)= A^{ij}(t-\tau_1,\tau_2)+A^{ij}(t,\tau_1),$$
for any $\tau_1,\tau_2>0$, and use $(iv)$.
\end{proof}

\subsection{Proof of Theorem \ref{fundamental 1}}\label{derivation}

By Lemma \ref{nucleo}$(i)$, for any $1\leq p\leq\infty$,
$$\|u\|_{L^p(\R^{n+1})}\leq\int_{\R^{n+1}}p(t,\tau,y)\|f\|_{L^p(\R^{n+1})}\,dy\,d\tau
=\|f\|_{L^p(\R^{n+1})}.$$

Assume next that $f\in C^2_c(\R^{n+1})$. Uniqueness of bounded classical solutions to \eqref{EQ1}
follows from \cite[Theorem~8.1.7]{Krylov book}.
Now, the argument we just performed above also shows that we can interchange the integral
and the second derivatives $\partial_{ij}$ to get
\begin{align*}
\partial_{ij}u(t,x) &= \iint_{\R^{n+1}}p(t,\tau,y)\partial_{x_ix_j}f(t-\tau,x-y)\,dy\ d\tau \\
&= \lim_{\varepsilon \rightarrow 0 }\iint_{\Omega_\varepsilon}p(t,\tau,y)\partial_{y_iy_j}f(t-\tau,x-y)\,dy\,d\tau.
\end{align*}
where $ \Omega_\varepsilon=\{(\tau,y):\max(\tau^{1/2},|y|)>\varepsilon\}$. Integration by parts gives
\begin{align*}
\iint_{\Omega_\varepsilon}p(t,\tau,y)&\partial_{y_iy_j}f(t-\tau,x-y)\,dy\,d\tau \\
&=\iint_{\partial{\Omega_\varepsilon}}p(t,\tau,y)\partial_{y_j}f(t-\tau,x-y)\nu_i\,dS(y,\tau)\\
&\quad  - \iint_{\Omega_\varepsilon}\partial_{y_i}p(t,\tau,y)\partial_{y_j}f(t-\tau,x-y)\,dy\,d\tau,
\end{align*}
where $\nu_i$ is the $i$th-component of the exterior unit normal vector to $\partial{\Omega_\varepsilon}$.
Let us write
\begin{equation}\label{eq:omega1}
\partial{\Omega_\varepsilon} = \partial{\Omega_\varepsilon^1}\cup
\partial{\Omega_\varepsilon^2}\cup\partial{\Omega_\varepsilon^3},
\end{equation}
where (remember that $\tau>0$)
\begin{equation}\label{eq:omega2}
\begin{aligned}
\partial{\Omega_\varepsilon^1}&=\{(\tau,y):  |y| <  \varepsilon,\,\tau^\frac{1}{2} = \varepsilon\},\\
\partial{\Omega_\varepsilon^2}&= \{(\tau,y):  |y| = \varepsilon, 0< \tau^\frac{1}{2} <\varepsilon \},\\
\partial{\Omega_\varepsilon^3}&= \{(\tau,y):  |y| < \varepsilon,\, \tau^\frac{1}{2} = 0 \}.
\end{aligned}
\end{equation}
The exterior unit normal vector on
$\partial{\Omega_\varepsilon^1}$ is $(-1,0,\ldots,0)\in\R^{n+1}$. Then
$$\iint_{\partial{\Omega_\varepsilon^1}}p(t,\tau,y)\partial_{y_j}f(t-\tau,x-y)\nu_i\,dS(y,\tau)= 0,$$
and the same is true for the boundary integral over $\Omega_\varepsilon^3$. 
On the other hand, the unit normal  of $\partial{\Omega_\varepsilon^2}$ is $\frac{1}{\varepsilon}(0,-y)$. Hence, by Lemma \ref{nucleo},
$$\iint_{\partial{\Omega_\varepsilon^2}}p(t,\tau,y)|\partial_{y_j}f(t-\tau,x-y)|\,dS(y,\tau)
\leq C\int_0^{\varepsilon^2}\frac{e^{-\varepsilon^2/(4\tau)}}{\tau^{n/2}}\varepsilon^{n-1}\,d\tau  
=C\varepsilon\rightarrow0,$$
as $\varepsilon\to0$.
Again, integration by parts together with an analogous discussion of the boundary integrals gives
\begin{align*}
-\iint_{\Omega_\varepsilon}\partial_{y_i}p(t,\tau,y)&\partial_{y_j}f(t-\tau,x-y)\,dy\,d\tau \\
&=\iint_{\Omega_\varepsilon}\partial_{y_iy_j}p(t,\tau,y)f(t-\tau,x-y)\,dy\,d\tau\\
&\quad-\iint_{\partial{\Omega^2_\varepsilon}}\partial_{y_i}p(t,\tau,y)f(t-\tau,x-y)\nu_j\,dS(y,\tau) \\
&=:I_1-I_2.
\end{align*}
The integral $I_1$ is the first term in \eqref{equation2 x}. Let us rewrite $I_2$ as
\begin{align*}
I_2 &=\iint_{\partial{\Omega_\varepsilon^2}}\partial_{y_i}p(t,\tau,y)
\big(f(t-\tau,x-y)-f(t,x)\big)\frac{y_j}{\varepsilon}\,dS(y,\tau)\\
&\quad+f(t,x)\iint_{\partial{\Omega_\varepsilon^2}}\partial_{y_i}p(t,\tau,y)\frac{y_j}{\varepsilon}\,dS(y,\tau) \\
&=: I_{21}+I_{22}.
\end{align*}
By Lemma \ref{nucleo} and the Mean Value  Theorem  we get 
\begin{align*}
|I_{21}| &\leq C\int_0^{\varepsilon^2} \int_{|y|=\varepsilon}
\frac{|y|e^{-|y|^2/(4\tau)}}{\tau^{n/2+1}}(\tau+|y|)\,dS(y)\,d\tau\leq
C \int_0^{\varepsilon^2}\frac{ \varepsilon^{n+1}e^{-\varepsilon^2/(4\tau)}}{\tau^{n/2+1}}\,d\tau
=C\varepsilon\rightarrow0,
\end{align*}
as $\varepsilon\to0$. Now the integral in $I_{22}$ depends on $\varepsilon$ and
$a^{ij}(t)$, so we call it $I^\varepsilon_{ij}(a)(t)$. By \eqref{firstd},
\begin{align*}
I^\varepsilon_{ij}(a)(t)&=-\frac{1}{2}\int_0^{\varepsilon^2}\int_{|y|=\varepsilon}
p(t,\tau,y)(B_{t,\tau} y)_i\frac{y_j}{\varepsilon}\,dS(y)\,d\tau \\
&=-\frac12\int_0^{\varepsilon^2}e^{-\tau}\int_{|y|=\varepsilon} 
\frac{\exp\big(-\tfrac{1}{4}\langle B_{t,\tau}y,y\rangle\big)}{(4\pi)^{n/2} (\det B_{t,\tau})^{-1/2}}\,
(B_{t,\tau} y)_i\frac{y_j}{\varepsilon}\,dS(y)\,d\tau \\
&=-\frac{1}{2}\int_0^{1}e^{-\varepsilon^2\tau}\int_{|y|=1} 
\frac{\exp\big(-\tfrac{\varepsilon^2}{4}\langle B_{t,\varepsilon^2\tau}y,y\rangle\big)}{(4\pi)^{n/2}
(\det(\varepsilon^2B_{t,\varepsilon^2\tau}))^{-1/2}}\,
(\varepsilon^2B_{t,\varepsilon^2\tau}y)_iy_j\,dS(y)\,d\tau.
\end{align*}
Observe that 
$\varepsilon^2 B_{t,\varepsilon^2 \tau}\varepsilon^{-2}A_{t,\varepsilon^2 \tau}=I$. Hence 
$\lim_{\varepsilon\rightarrow 0} \varepsilon^2 B_{t,\varepsilon^2 \tau}
= \big(\lim_{\varepsilon\rightarrow 0} \varepsilon^{-2} A_{t,\varepsilon^2 \tau}\big)^{-1}$. But 
$$\lim_{\varepsilon\rightarrow 0} \varepsilon^{-2}A_{t,\varepsilon^2 \tau}= \tau\lim_{\varepsilon\rightarrow 0}
\frac1{\tau\varepsilon^2} \int_{t-\varepsilon^2 \tau}^ta(r)\,dr= \tau a(t),$$
for a.e.~$t$. Hence, by taking the limit as $\varepsilon\to0$ in $I^\varepsilon_{ij}(a)(t)$, performing
the change of variables $1/\tau=r^2$, and using polar coordinates,
\begin{align*}
I_{ij}(a)(t) &= \lim_{\varepsilon\to0} I^\varepsilon_{ij}(a)(t) \\
&= -\frac{1}{2}\int_0^{1} \int_{|y|=1} 
\frac{\exp\big(-\tfrac{1}{4\tau}\langle a(t)^{-1}y,y\rangle\big)}{\tau^{n/2+1}(4\pi)^{n/2}(\det a(t))^{1/2}}\,
(a(t)^{-1}y)_iy_j\,dS(y)\,d\tau \\
&=\int_1^\infty\int_{|y|=1} 
\frac{\exp\big(-\tfrac{1}{4}\langle a(t)^{-1}ry,ry\rangle\big)}{(4\pi)^{n/2}(\det a(t))^{1/2}}\,
(a(t)^{-1}y)_iy_j\,dS(y)r^{n-1}\,dr \\
&=\int_{\{x: |2 a(t)^{1/2}x|\ge 1\}}\frac{e^{-|x|^2}}{\pi^{n/2}}\,\frac{(a(t)^{-1/2}x)_i (a(t)^{1/2} x)_j}{ |a(t)^{1/2} x|^2}\,dx.
\end{align*}
This finishes the proof of \eqref{equation2 x} and \eqref{matrizA}.

Now we compute $\partial_tu(t,x)$. In a similar fashion as before,
\begin{align*}
\partial_tu(t,x) &= \lim_{\varepsilon\rightarrow 0}
\Bigg[\iint_{\Omega_\varepsilon}\partial_tp(t,\tau,y) f(t-\tau,x-y)\,dy\,d\tau \\
&\qquad\qquad-\iint_{\Omega_\varepsilon}p(t,\tau,y)
\partial_\tau f(t-\tau,x-y)\,dy\,d\tau\Bigg]\\
&=\lim_{\varepsilon\rightarrow 0}\iint_{\Omega_\varepsilon}(\partial_t+\partial_\tau)p(t,\tau,y)f(t-\tau,x-y)\,dy\,d\tau \\
&\quad-\lim_{\varepsilon\rightarrow 0}\iint_{\partial{\Omega_\varepsilon}}p(t,\tau,y)f(t-\tau,x-y)\nu_\tau\,dS(y,\tau).
\end{align*}
Again, we decompose $\partial\Omega_\varepsilon$ as in \eqref{eq:omega1}--\eqref{eq:omega2}.
Clearly, $\nu_\tau=0$ on $\partial\Omega^2_\varepsilon$. On the other hand,
$$\iint_{\partial{\Omega^3_\varepsilon}}p(t,\tau,y)f(t-\tau,x-y)\,dS(y,\tau)
=\int_{|y|=\varepsilon}p(t,\tau,y)f(t-\tau,x-y)\,dS(y)=0.$$
Parallel to the spatial derivatives case we write
\begin{align*}
\iint_{\partial{\Omega^1_\varepsilon}}&p(t,\tau,y)f(t-\tau,x-y)\nu_\tau\,dS(y,\tau) \\
&= \iint_{\partial{\Omega^1_\varepsilon}}p(t,\tau,y)\big(f(t-\tau,x-y)-f(t,x)\big)\nu_\tau\,dS(y,\tau) \\
&\quad+f(t,x)\iint_{\partial{\Omega^1_\varepsilon}}p(t,\tau,y)\nu_\tau\,dS(y,\tau)\\
&=:J_1^\varepsilon+J_2^\varepsilon.
\end{align*}
Apply the Mean Value Theorem in $J_1^\varepsilon$ and Lemma \ref{nucleo} to get
\begin{align*}
|J_1^\varepsilon| &\leq C \int_{\tau=\varepsilon^2}\int_{|y|< \varepsilon} \frac{e^{-|y|^2/(4\tau)}}{\tau^{n/2}}(\tau+|y|)\,dy\,d\tau \\
&\leq \frac{C}{\varepsilon^{n-1}}\int_0^\varepsilon r^{n-1}e^{-r^2/(4\varepsilon^2)}\,dr
= C\varepsilon\rightarrow 0,
\end{align*}
as $\varepsilon\rightarrow0$, where we have assumed that $\varepsilon<1$. 
On the other hand, by a limit argument similar to the one used for $\partial_{ij}u$ before,
\begin{align*}
J_2^\varepsilon &= f(t,x)\int_{|y|<\varepsilon}e^{-\varepsilon^2}
\frac{\exp\big(-\tfrac{1}{4}\langle B_{t,\varepsilon^2}y,y\rangle\big)}{(4\pi)^{n/2} (\det B_{t,\varepsilon^2})^{-1/2}}\,dy \\
&=f(t,x)\int_0^1\int_{|y|=1}e^{-\varepsilon^2}
\frac{\exp\big(-\tfrac{\varepsilon^2}{4}\langle B_{t,\varepsilon^2}ry,ry\rangle\big)}{(4\pi)^{n/2}
(\det(\varepsilon^2B_{t,\varepsilon^2}))^{-1/2}}\,dS(y)r^{n-1}\,dr \\
&\longrightarrow f(t,x)\int_0^1\int_{|y|=1}
\frac{\exp\big(-\tfrac{1}{4}\langle a(t)^{-1}ry,ry\rangle\big)}{(4\pi)^{n/2}
(\det a(t))^{1/2}}\,dS(y)r^{n-1}\,dr \\
&\qquad =f(t,x)\int_{\{x: |2a(t)^{1/2}x|\le 1\}}\frac{e^{-|x|^2}}{\pi^{n/2}}\,dx=f(t,x)J(a)(t).
\end{align*}
This proves \eqref{equation2 t} and \eqref{matrizb}.

We finally check that $\partial_tu-a^{ij}(t)\partial_{ij}u+u=f$. Indeed, since
$$a^{ij}(t)\frac{(a(t)^{-1/2}x)_j (a(t)^{1/2} x)_i}{ |a(t)^{1/2} x|^2}=1$$
and $\displaystyle\int_{\mathbb{R}^n}e^{-|x|^2}\,dx=\pi^{n/2}$,
it is clear that
$$J(a)(t)+a^{ij}(t)I_{ij}(a)(t)=1.$$
The conclusion follows from Lemma \ref{nucleo}$(ii)$.
This completes the proof of Theorem \ref{fundamental 1}.\qed

\subsection{Proof of Theorem \ref{fundamental 3}}\label{subsection:cauchy}

The $L^p$ estimates of $v$ in terms of $L^p$ norms of $f$ and $g$ follow like in the case of
$u$ and by using Lemma \ref{nucleo}.

As with Theorem \ref{fundamental 1}, uniqueness is a consequence of \cite[Theorem~8.1.7]{Krylov book}.

For the computation of derivatives, by linearity, it is enough to study the problems 
$$\begin{cases} 
\partial_t\varphi-a^{ij}(t)\partial_{ij}\varphi+\varphi=0,&\hbox{for}~t >0,~x\in\R^n,\\
\varphi(0,x)=g(x),&\hbox{for}~x\in\mathbb{R}^n,
\end{cases}$$
and
$$\begin{cases} 
\partial_t\psi-a_{ij}(t)\partial_{ij}\psi+\psi=f,&\hbox{for}~t >0,~x\in\R^n,\\
\psi(0,x)=0,&\hbox{for}~x\in\mathbb{R}^n,
\end{cases}$$
separately and then take $v=\varphi+\psi$.

On one hand, the solution $\varphi$ is given by
$$\varphi(t,x)=\int_{\R^n}p(t,t,y)g(x-y)\,dy,\quad t\geq0,~x\in\R^n.$$
It can be directly checked that this produces all the terms and properties
in the statement related to the initial datum $g$.

For the second problem, the solution is
$$\psi(t,x)=\int_0^t\int_{\R^n}p(t,\tau,y)f(t-\tau,x-y)\,dy\,d\tau.$$
Indeed, clearly $\psi(0,x)=0$. To derive the formulas for the derivatives,
we proceed as in the proof of Theorem \ref{fundamental 1} but
with the appropriate changes due to the nature of the new ambient space $\mathbb{R}^{n+1}_+$.
We sketch the computation here and leave details to the interested reader.
We begin as in the proof of Theorem \ref{fundamental 1}, but replacing the set
$\Omega_\varepsilon$ by the set $\Sigma_\varepsilon = \{(\tau,y)\in\R^{n+1}_+:\tau>\varepsilon\}$.
Using integration by parts twice in space we obtain the first term in formula \eqref{equation3 x}.
For the derivative with respect to $t$, we notice that
$\partial \Sigma_\varepsilon= \{(\tau,y)\in\R^{n+1}_+:\tau=\varepsilon\}$.
Then parametric derivation and integration by parts yields 
\begin{align*}
\partial_t\psi(t,x) &= \lim_{\varepsilon\to0}\int_\varepsilon^t\int_{\mathbb{R}^n}
(\partial_t+\partial_\tau )p(t,\tau,y)f(t-\tau,x-y)\,d\tau\,dy \\
&\quad+\lim_{\varepsilon\to0}\int_{\mathbb{R}^n}p(t,\varepsilon,y)f(t-\varepsilon,x-y)\,dy
- \int_{\mathbb{R}^n}p(t,t,y)f(0,x-y)\,dy \\
&= \lim_{\varepsilon\to0}\int_\varepsilon^t \int_{\mathbb{R}^n}(\partial_t+\partial_\tau) p(t,\tau,y) f(t-\tau,x-y)\,d\tau\,dy\\
&\quad+\lim_{\varepsilon \rightarrow 0}\int_{\mathbb{R}^n}p(t,\varepsilon,y)\big(f(t-\varepsilon,x-y)-f(t,x-y)\big)\,dy\\
&\quad+\lim_{\varepsilon \rightarrow 0}  \int_{\mathbb{R}^n} p(t,\varepsilon,y)f(t,x-y)\,dy\\
&=  \lim_{\varepsilon \rightarrow 0} \int_\varepsilon^t \int_{\mathbb{R}^n}(\partial_t+ \partial_\tau) p(t,\tau,y)f(t-\tau,x-y)\,d\tau\,dy+f(t,x).
\end{align*}
The proof of Theorem \ref{fundamental 3} is complete.\qed

\section{Weighted vector-valued Calder\'on--Zygmund singular integrals \\ on spaces of homogeneous type}\label{Section:eth}

Let $\X$ be a set. A function $\rho:\X\times\X\to[0,\infty)$
is called a quasidistance in $\X$ if for any $\x,\y,\z\in\X$ we have: (1) $\rho(\x,\y)=0$ if and only if $\x=\y$,
(2) $\rho(\x,\y)=\rho(\y,\x)$, and (3) $\rho(\x,\z)\le\kappa(\rho(\x,\y)+\rho(\y,\z))$ for some constant $\kappa\geq1$.
We assume that $\X$ has the topology induced by 
the open balls $B(\x,r)$ with center at $\x\in\X$ and radius $r>0$ defined as $B(\x,r):=\{\y\in\X:\rho(\x,\y)<r\}$.
Let $\mu$ be a positive Borel measure on $(\X,\rho)$
such that, for some universal constant $C_d>0$, we have $\mu(B(\x,2r)) \le C_d\mu(B(\x,r))$ (the so-called doubling property),
for every $\x\in\X$ and $r >0$.
Then $(\X,\rho,\mu)$ is called a \textit{space of homogeneous type}.

It is clear that $\R^n$ with the usual Euclidean distance and the Lebesgue measure is a space 
of homogeneous type. See Remark \ref{rem:parabolic metric space} for the example of the parabolic metric space.

Let $w:\X\to\R$ be a weight, namely, a measurable function such that $w(\x)>0$
for $\mu$-a.e. $\x\in\X$.
Given a Banach space $E$, we denote by $L^p_E(\X,w)=L^p(\X,w;E)$, $1\leq p\leq\infty$, the space
of strongly measurable $E$-valued
functions $f$ defined on $\X$ such that $\|f\|_E$ belongs to $L^p(\X,w(\x)d\mu)$.
When $w=1$ we just write $L^p_E(\X)=L^p(\X;E)$. The norms are given by
$$\|f\|_{L^p(\X,w;E)}^p=\int_{\X}\|f(\x)\|_{E}^pw(\x)\,d\mu,\quad\hbox{when}~1\leq p<\infty,$$
and
$$\|f\|_{L^\infty(\X,w;E)}=\|f\|_{L^\infty(\X;E)}=\operatornamewithlimits{ess\,sup}_{\x\in\X}\|f(\x)\|_{E},$$
where the supremum is taken with respect to $\mu$. We use the notation
$$L^\infty_{E,c}(\X)=L^\infty_{c}(\X;E)=\{f\in L^\infty(\X;E):\supp(f)~\hbox{is compact in}~\X\}.$$

\begin{defn}[Vector-valued Calder\'on--Zygmund operator on $(\X,\rho,\mu)$]\label{CZ}
Let $E,F$ be Banach spaces. We say that a linear  operator $T$ on 
a space of homogeneous type $(\X,\rho,\mu)$ is a
(vector-valued) Calder\'on--Zygmund operator if it satisfies the following conditions.
\begin{itemize}
\item [(I)] Either there exists $1\le p_0 <\infty$ such that $T$ is bounded from
$L^{p_0}(\X;E)$ into $L^{p_0}(\X;F)$; or $T$ is bounded from $L^\infty_c(\X;E)$ into $L^\infty(\X;F)$.
\item[(II)] For  bounded $E$-valued functions $f$ with compact support, $Tf$ can be represented as 
\begin{equation}\label{CZ2}
Tf(\x) = \int_{\X}K(\x,\y)f(\y)\,d\mu, \quad\hbox{for}~\x\notin{\rm supp}(f),
\end{equation}
where, for fixed $\x,\y\in\X$ such that $\x\neq\y$,
the kernel $K(\x,\y)$ belongs to $\mathcal{L} (E,F)$, the space of bounded linear operators from $E$ to $F$ and, moreover, 
satisfies the following estimates:
\begin{itemize}
\item[(II.1)]  $\displaystyle \|K(\x,\y) \|\le  \frac{C}{\mu(B(\x,\rho(\x,\y))}$, for every $\x\neq\y$;
\item[(II.2)] $\displaystyle   \|K(\x,\y) - K(\x,\y_0) \|+\|K(\y,\x) - K(\y_0,\x)\|
\le C\frac{\rho(\y,\y_0)}{\rho(\x,\y_0)\mu(B(\y_0,\rho(\x,\y_0))},$
whenever $\rho(\x,\y_0)>2\rho(\y,\y_0)$;
\end{itemize}
where $\|\cdot\|$ denotes the operator norm, for some constant $C>0$.
\end{itemize}
\end{defn}

A weight $w$ on $(\X,\rho,\mu)$ is a Muckenhoupt $A_p(\X)$ weight, $1<p<\infty$,
if there exists a constant $C_w>0$ such that
$$\bigg(\frac{1}{\mu(B)}\int_Bw(\x)\,d\mu\bigg)\bigg(\frac{1}{\mu(B)}\int_Bw(\x)^{1/(1-p)}\,d\mu\bigg)^{p-1}\leq C_w,$$
for every metric ball $B\subset\X$. We say that $w\in A_1(\X)$ if there is a constant $C_w>0$ such that
$$\frac{1}{\mu(B)}\int_{B}w(\y)\,d\mu(\y)\leq C_w\inf_{\x\in B}w(\x),$$
for every metric ball $B\subset\X$.

Let $(F,\|\cdot\|_F)$ be a Banach space. The $F$-sharp maximal function $M_F^\#$ is given by
\begin{equation}\label{eq:sharp maximal}
M^\#_Fg(\x):=\sup_{B\subset\X}\frac{1}{\mu(B)}\int_B\|g(\y)-g_B\|_F\,d\mu,
\end{equation}
where $g$ is a locally integrable function on $(\X,\rho,\mu)$ with values in $F$.
The supremum above is taken over
all metric balls $B\subset\X$ that contain the point $\x$, and $\displaystyle g_B:=\frac{1}{\mu(B)}\int_Bg(\y)\,d\mu$.
We say that $g$ is in the $F$-valued $BMO_F(\X)=BMO(\X;F)$ space if
$$\|g\|_{BMO(\X;F)}=\|M_F^\# g\|_{L^\infty(\X)}<\infty.$$

\begin{thm}[Calder\'on--Zygmund Theorem]\label{thm:CZ}
If $T$ is a vector-valued Calder\'on--Zygmund
operator on a space of homogeneous type $(\X,\rho,\mu)$ as defined above
then $T$ extends as a bounded operator
\begin{itemize}
\item[$(a)$] from $L^p(\X,w;E)$ into $L^p(\X,w;F)$, for any $1<p<\infty$ and $w\in A_p(\X)$;
\item[$(b)$] from $L^1(\X,w;E)$ into weak-$L^1(\X,w;F)$, for any $w\in A_1(\X)$; and
\item[$(c)$] from $L^\infty_{c}(\X;E)$ into $BMO(\X;F)$.
\end{itemize}
Moreover, the maximal operator of the truncations defined by
\begin{equation}\label{truncation}
T^\ast f(\x) =\sup_{\varepsilon>0}\|T_\varepsilon f(\x)\|_F=
\sup_{\varepsilon>0}\bigg\| \int_{\rho(\x,\y)>\varepsilon} K(\x,\y)f(\y)\,d\mu\bigg\|_F
\end{equation}
is a bounded operator
\begin{itemize}
\item[$(d)$] from $L^p(\X,w;E)$ into $L^p(\X,w)$, for any $1<p<\infty$ and $w\in A_p(\X)$; and
\item[$(e)$] from $L^1(\X,w;E)$ into weak-$L^1(\X,w)$, for any $w\in A_1(\X)$.
\end{itemize}
In particular, the set
$$\Big\{f\in L^p(\X,w;E):\lim_{\varepsilon\to0^+}T_\varepsilon f(\x)
~\hbox{exists in}~F~\hbox{for}~\mu\hbox{-a.e.}~\x\in\X\Big\},$$
is closed in $L^p(\X,w;E)$, for every $1\leq p<\infty$.
\end{thm}

For full details about the theory presented above see
\cite{Calderon-Zygmund,Duo,MST, RRT, Francisco, RT}.

Next we generalize a result from \cite{Segovia-Torrea} to the context of vector-valued
Calder\'on--Zygmund operators on spaces of homogeneous type. Observe that the following
statement generalizes part $(c)$ of Theorem \ref{thm:CZ} to include weights.
The proof will be provided in the Appendix.

\begin{thm}[Segovia--Torrea]\label{thm:Segovia-Torrea}
Let $T$ be a vector-valued Calder\'on--Zygmund operator
on a space of homogeneous type $(\X,\rho,\mu)$ as defined above.
Suppose that $w>0$ is a weight such that $w^{-1}\in A_1(\X)$. Then,
for every $f\in L^\infty_c(\X;E)$ such that $wf\in L^\infty(\X;E)$, we have
$$\|wM_F^\# (Tf)\|_{L^\infty(\X)}\leq C_{w,T,\rho,\mu}\|wf\|_{L^\infty(\X;E)},$$
where $M_F^\#$ is the $F$-sharp maximal function \eqref{eq:sharp maximal}.
\end{thm}

\begin{rem}[Parabolic metric space]\label{rem:parabolic metric space}
The parabolic metric space is the space of homogeneous type $(\X,\rho,\mu)=(\mathbb{R}^{n+1},\rho,dtdx)$,
where $\rho$ is the parabolic distance defined by
\begin{equation}\label{eq:parabolic distance}
\rho\big((t,x),(\tau,y)\big)=\max(|t-\tau|^{1/2},|x-y|),\quad\hbox{for}~(t,x),(\tau,y)\in\R^{n+1},
\end{equation}
and $dtdx$ is the Lebesgue measure
on $\R^{n+1}$. The parabolic ball centered at $(t,x)\in\R^{n+1}$
with radius $r>0$ is given by $B((t,x),r)=\{(\tau,y)\in\R^{n+1}:\max(|t-\tau|^{1/2},|x-y|)<r\}$. Then,
$$|B((t,x),r)|=|B((0,0),r)|\sim r^{n+2},$$
so $dtdx$ is a doubling measure.
Notice next that for the case of the parabolic distance \eqref{eq:parabolic distance} the right hand sides in conditions
(II.1) and (II.2) read, for $\x=(t,x)$, $\y=(\tau,y)$ and $\y_0=(\tau_0,y_0)$,
$$\frac{C}{\mu(B(\x,\rho(\x,\y))}\sim\frac{C}{(|t-\tau|^{1/2}+|x-y|)^{n+2}},$$
and, whenever $|t-\tau_0|^{1/2}+|x-y_0|>2(|\tau-\tau_0|^{1/2}+|y-y_0|)$,
$$\frac{\rho(\y,\y_0)}{\rho(\x,\y_0)\mu(B(\y_0,\rho(\x,\y_0))}\sim
\frac{|\tau-\tau_0|^{1/2}+|y-y_0|}{(|t-\tau_0|^{1/2}+|x-y_0|)^{n+3}},$$
respectively. The set of points $\y\in\X$ such that $\rho(\x,\y)>\varepsilon$ that appears in \eqref{truncation} is
\begin{equation}\label{conjuntitos}
\Omega_\varepsilon(t,x):=\{(\tau,y)\in\R^{n+1}:\max(|t-\tau|^{1/2},|x-y|)>\varepsilon\}.
\end{equation}
The class of Muckenhoupt $A_p$ weights $w=w(t,x)$
in this particular case is called the parabolic $A_p$ class
and will be denoted by $PA_p(\R^{n+1})$, $1\leq p<\infty$. The maximal function in the parabolic
metric space is
$$\mathcal{M}f(t,x)=\sup_{r>0}\frac{1}{|B((t,x),r)|}\int_{B((t,x),r)}|f(\tau,y)|\,d\tau\,dy$$
where $B((t,x),r)$ are the parabolic balls as defined above. It is known that $\mathcal{M}$ is bounded in
$L^p(\R^{n+1},w)$, for $1<p<\infty$ and $w\in PA_p(\R^{n+1})$, and from $L^1(\R^{n+1},w)$
into weak-$L^1(\R^{n+1},w)$, for $w\in PA_1(\R^{n+1})$ (see \cite{Calderon}, also \cite{RT}).
Moreover, $w\in PA_1(\R^{n+1})$ if and only
if there exists $C>0$ such that $\mathcal{M}w(t,x)\leq Cw(t,x)$, for a.e. $(t,x)\in\R^{n+1}$.
\end{rem}

\section{Parabolic weighted Sobolev estimates:
proofs of Theorems \ref{thm:parabolic weights} and \ref{thm:parabolic-BMO}.}

\subsection{Proof of Theorem \ref{thm:parabolic weights}}

We begin with the proof of \eqref{eq:uweights}. From the first estimate in
Lemma \ref{nucleo}$(iii)$, and by taking into account the cases
$|y|^2/\tau>1$ and $|y|^2/\tau\leq1$, it can be readily seen that
given any $N>0$, there exist constants $C_N,c_N>0$ such that
$$p(t,\tau,y)\leq C_N\frac{e^{-c_N|\tau|}}{(\tau^{1/2}+|y|)^{n+N}}$$
for any $t,\tau\in\R$, $y\in\R^n$. Consequently, $p(t,\tau,y)$ can be bounded by a nonnegative,
radially decreasing, integrable function in the parabolic metric space $\R^{n+1}$. Therefore, there
exists $C_n>0$ such that
\begin{equation}\label{eq:uleqM}
|u(t,x)|\leq C_n\mathcal{M}f(t,x).
\end{equation}
Now, \eqref{eq:uweights} holds for $p>1$ because $\mathcal{M}$ is bounded
in $L^p(\R^{n+1},w)$, with $w\in PA_p(\R^{n+1})$. For the case $p=1$, if $w\in PA_1(\R^{n+1})$,
and exactly as was done in \eqref{eq:uleqM},
\begin{align*}
\int_{\R^{n+1}}|u(t,x)|w(t,x)\,dt\,dx &\leq \int_{\R^{n+1}}|f(\tau,y)|\bigg[\int_{\R^{n+1}}p(t,t-\tau,x-y)w(t,x)\,dt\,dx
\bigg]\,d\tau\,dy \\
&\leq C\int_{\R^{n+1}}|f(\tau,y)|\mathcal{M}w(\tau,y)\,d\tau\,dy \\
&\leq C\int_{\R^{n+1}}|f(\tau,y)|w(\tau,y)\,d\tau\,dy = C\|f\|_{L^1(\R^{n+1},w)}.
\end{align*}

Next, let us prove that the operators $R_{ij}$, $i,j=1,\ldots,n$, defined by
$$R_{ij}f(t,x)=\lim_{\varepsilon\rightarrow0^+}\int_{\Omega_\varepsilon}
\partial_{y_iy_j}p(t,\tau,y)f(t-\tau,x-y)\,dy\,d\tau-f(t,x)I_{ij}(a)(t)$$
are parabolic Calder\'on--Zygmund operators according to Definition \ref{CZ},
with Banach spaces $E=F=\R$ and $p_0=2$.

Let $f\in L^\infty_c(\R^{n+1})$. By using the convolution property of the Fourier transform on $\R^n$,
Lemma \ref{nucleo} and Plancherel's Theorem on $\R^{n+1}$,
\begin{align*}
\bigg\|\int_\R\int_{\R^n}&\partial_{y_iy_j} p(t,\tau,y) f(t-\tau,x-y)\,dy\,d\tau\bigg\|^2_{L^2(\mathbb{R}^{n+1})} \\
&= \bigg\|\int_{\R}\xi_i\xi_j\widehat{p(t,\tau,\cdot)}(\xi)\widehat{f(t-\tau,\cdot)}(\xi)\,d\tau\bigg\|^2_{L^2(\mathbb{R}^{n+1})} \\ &= \bigg\|\int_{\R}\xi_i\xi_j\chi_{\tau>0}e^{-\tau}\frac{1}{(4\pi)^{n/2}}e^{-\langle A_{t,\tau} \xi,\xi\rangle}
\widehat{f(t-\tau,\cdot)}(\xi)\,d\tau\bigg\|^2_{L^2(\mathbb{R}^{n+1})} \\
&\le C\bigg\|\int_{\R}\chi_{\tau>0}|\xi|^2e^{-\Lambda\tau|\xi|^2}|
\widehat{f(t-\tau,\cdot)}(\xi)|\,d\tau\bigg\|^2_{L^2(\R^{n+1})} \\
&\leq C\bigg\|\bigg(\int_{\R}\chi_{\tau>0}|\xi|^2e^{-\Lambda\tau|\xi|^2}|
\widehat{f(t-\tau,\cdot)}(\xi)|^2\,d\tau\bigg)^{1/2}\\
&\qquad\quad\times\bigg(\int_0^\infty|\xi|^2e^{-\Lambda\tau|\xi|^2}\,d\tau\bigg)^{1/2}
\bigg\|^2_{L^2(\R^{n+1})} \\
&= C\int_{\R^{n+1}}\int_{\R}\chi_{\tau>0}|\xi|^2  e^{-\Lambda\tau |\xi|^2}|\widehat{f(t-\tau,\cdot)}(\xi)|^2\,d\tau\,d\xi\,dt \\
&= C\int_{\R^{n+1}}\int_{\R} \chi_{t>s} |\xi|^2  e^{-\Lambda(t-s)|\xi|^2}  |\widehat{f(s,\cdot)}(\xi)|^2\,dt\,d\xi\,ds \\
&= C\int_{\R^{n+1}}\bigg(\int_0^\infty|\xi|^2  e^{-\Lambda r|\xi|^2}\,dr\bigg)|\widehat{f(s,\cdot)}(\xi)|^2\,d\xi\,ds \\
&= C_{n,\Lambda}\int_{\mathbb{R}^{n+1}}|\widehat{f(s,\cdot)}(\xi)|^2\,d\xi\,ds
=C_{n,\Lambda}\| f \|^2_{L^2(\mathbb{R}^{n+1})}.
\end{align*}
This computation shows that the integral operator in the right hand side of formula \eqref{equation2 x}
extends as a bounded operator from $L^{p_0}(\R^{n+1})$ into itself, for $p_0=2$. On the other hand, by 
making the change of variables $y=a(t)^{1/2}x$ and using the ellipticity of $a(t)$,
\begin{align*}
|I_{ij}(a)(t)| &\leq \int_{\{x:|2a(t)^{1/2}x|>1\}}\frac{e^{-|x|^2}}{\pi^{n/2}}
\frac{|a(t)^{-1/2}x||a(t)^{1/2}x|}{|a(t)^{1/2}x|^2}\,dx \\
&=  \int_{\{x:|2y|>1\}}\frac{e^{-|a^{-1/2}(t)y|^2}}{\pi^{n/2}}
\frac{|a(t)^{-1}y|}{|y|}|\det(a^{-1/2}(t))|\,dx \\
&\leq C_{n,\Lambda}\int_{\R^n}e^{-\Lambda|y|^2}\,dy=C_{n,\Lambda}.
\end{align*}
Hence $I_{ij}(a)(t)\in L^\infty(\R^{n+1})$. Therefore, the operators $R_{ij}$ extend as bounded operators
on $L^2(\R^{n+1})$.

Clearly, for any $f\in L^\infty_c(\R^{n+1})$ we have the representation
$$R_{ij}f(t,x)=\iint_{\R^{n+1}}K_{ij}((t,x),(\tau,y))f(\tau,y)\,d\tau\,dy,\quad\hbox{for}~(t,x)\notin\supp(f),$$
where the kernel $K_{ij}$ is given by
\begin{equation}\label{eq:kernelRiesztransformsx}
K_{ij}((t,x),(\tau,y))\equiv K_{ij}(t,\tau,x-y)=\partial_{y_iy_j}p(t,t-\tau,x-y).
\end{equation}

The size and smoothness estimates for the kernel in \eqref{eq:kernelRiesztransformsx}
follow from Lemma \ref{nucleo}. Indeed,
\begin{align}\label{eq:tamanodeKij}
|K_{ij}&((t,x),(\tau,y))| \nonumber\\
&\leq C\chi_{t>\tau}e^{-(t-\tau)}\frac{e^{-|x-y|^2/(c(t-\tau))}}{((t-\tau)^{1/2})^{n+2}} \\
&\leq C\chi_{t>\tau}\frac{((t-\tau)^{1/2}+|x-y|)^{n+2}e^{-|x-y|^2/(c(t-\tau))}}{((t-\tau)^{1/2})^{n+2}}
\cdot\frac{1}{((t-\tau)^{1/2}+|x-y|)^{n+2}} \nonumber\\
&\leq C\chi_{t>\tau}\bigg(1+\frac{|x-y|}{(t-\tau)^{1/2}}\bigg)^{n+2}e^{-|x-y|^2/(c(t-\tau))}
\cdot\frac{1}{((t-\tau)^{1/2}+|x-y|)^{n+2}} \nonumber\\
&\leq \frac{C}{|t-\tau|^{1/2}+|x-y|)^{n+2}}.\nonumber
\end{align}
In a similar way, by Lemma \ref{nucleo},
\begin{align*}
|\partial_{x_k}K_{ij}((t,x),(\tau,y))|+|\partial_{y_k}K_{ij}((t,x),(\tau,y))|
&\leq C\chi_{t>\tau}e^{-c(t-\tau)}\frac{e^{-|x-y|^2/(c(t-\tau))}}{((t-\tau)^{1/2})^{n+3}} \\
&\leq \frac{C}{(|t-\tau|^{1/2}+|x-y|)^{n+3}},
\end{align*}
and
\begin{align}\label{eq:tamanioderivadasent}
|\partial_tK_{ij}((t,x),(\tau,y))|+|\partial_\tau K_{ij}((t,x),(\tau,y))| &\leq C\chi_{t>\tau}e^{-c(t-\tau)}
\frac{e^{-|x-y|^2/(c(t-\tau))}}{((t-\tau)^{1/2})^{n+4}} \\
&\leq \frac{C}{(|t-\tau|^{1/2}+|x-y|)^{n+4}}.\nonumber
\end{align}
Let $|t-\tau_0|^{1/2}+|x-y_0|\geq 2(|\tau-\tau_0|^{1/2}+|y-y_0|)$.
If $(s,\theta)$ is an intermediate point between $(\tau,y)$
and $(\tau_0,y_0)$ then $|t-s|^{1/2}+|x-\theta|\geq C(|t-\tau_0|^{1/2}+|x-y_0|)$.
By using the Mean Value Theorem,
\begin{align*}
|K_{ij}&((t,x),(\tau,y))-K_{ij}((t,x),(\tau_0,y_0))| \\
&=|\nabla_{\tau,y}K_{ij}((t,x),(s,\theta))\cdot(\tau-\tau_0,y-y_0)| \\
&\leq |\nabla_\tau K_{ij}((t,x),(s,\theta))||\tau-\tau_0|+|\nabla_yK_{ij}((x,t),(s,\theta))||y-y_0| \\
&\leq C\bigg(\frac{|\tau-\tau_0|}{(|t-s|^{1/2}+|x-\theta|)^{n+4}}+\frac{|y-y_0|}{(|t-s|^{1/2}+|x-\theta|)^{n+3}}\bigg) \\
&\leq C\bigg(\frac{|\tau-\tau_0|^{1/2}(|t-\tau_0|^{1/2}+|x-y_0|)}{(|t-\tau_0|^{1/2}+|x-y_0|)^{n+4}}
+\frac{|y-y_0|}{(|t-\tau_0|^{1/2}+|x-y_0|)^{n+3}}\bigg) \\
&=C\frac{|\tau-\tau_0|^{1/2}+|y-y_0|}{(|t-\tau_0|^{1/2}+|x-y_0|)^{n+3}}.
\end{align*} 
In a completely analogous way we can prove, under the condition that
$|t-\tau_0|^{1/2}+|x-y_0|\geq 2(|\tau-\tau_0|^{1/2}+|y-y_0|)$, the estimate
$$|K_{ij}((\tau,y),(t,x))-K_{ij}((\tau_0,y_0),(t,x))|\leq C\frac{|\tau-\tau_0|^{1/2}+|y-y_0|}{(|t-\tau_0|^{1/2}+|x-y_0|)^{n+3}}.$$
In other words the kernel $K_{ij}$ satisfies the size and smoothness conditions (I.2) and (II.2) of Definition \ref{CZ}.
Therefore, the conclusions on $\partial_{ij}u$ of the statement follow from Theorem \ref{thm:CZ}.

For the time derivative $f\longmapsto\partial_tu$, we can proceed in a parallel way
so that it shares the same boundedness properties as $R_{ij}f$.
Details are left to the interested reader.

The proof of Theorem \ref{thm:parabolic weights} is completed.\qed

\subsection{Proof of Theorem \ref{thm:parabolic-BMO}.}

In view of the proof of Theorem \ref{thm:parabolic weights},
Theorem \ref{thm:parabolic-BMO} is a direct
corollary of Theorem \ref{thm:Segovia-Torrea}.\qed

\section{Weighted mixed-norm Sobolev estimates: proofs of Theorems
\ref{thm:LqSobolev} and \ref{thm:mixed-BMO}}

\subsection{Proof of Theorem \ref{thm:LqSobolev}.}

Let us begin by proving \eqref{eq:umixedweights}. Let $M_{\R^{n}}$ denote the classical Hardy--Littlewood maximal
operator on $\R^n$, $n\geq1$. Recall that the classical heat semigroup
$$e^{r\Delta}f(x)=\frac{1}{(4\pi r)^{n/2}}\int_{\R^n}e^{-|x-y|^2/(4r)}f(y)\,dy$$
which is defined for $r>0$ and $x\in\R^n$, satisfies the pointwise inequality
$$\sup_{r>0}|e^{r\Delta}f(x)|\leq CM_{\R^n}f(x),$$
for some constant $C>0$, for every $x\in\R^n$.
Then, by the first estimate in Lemma \ref{nucleo}$(iii)$,
$$|u(t,x)|\leq C\int_{-\infty}^te^{-(t-\tau)}M_{\R^n}[f(\tau,\cdot)](x)\,d\tau.$$
Thus, if $p>1$,
\begin{align*}
\|u(t,\cdot)\|_{L^p(\R^n,\omega)} &\leq C\int_{-\infty}^te^{-(t-\tau)}\|M_{\R^n}[f(\tau,\cdot)]\|_{L^p(\R^n,\omega)}\,d\tau \\
&\leq C\int_{-\infty}^te^{-(t-\tau)}\|f(\tau,\cdot)\|_{L^p(\R^n,\omega)}\,d\tau \\
&\leq CM_{\R}[\|f(\cdot,\cdot)\|_{L^p(\R^n,\omega)}](t).
\end{align*}
In particular, if $q>1$,
$$\|u\|_{L^q(\R,\nu;L^p(\R^n,\omega))}\leq C\|M_{\R}[\|f(\cdot,\cdot)\|_{L^p(\R^n,\omega)}]\|_{L^q(\R,\nu)}
\leq C\|f\|_{L^q(\R,\nu;L^p(\R^n,\omega))},$$
while if $q=1$,
\begin{align*}
\int_\R\|u(t,\cdot)\|_{L^p(\R^n,\omega)}\nu(t)\,dt &\leq
C\int_{\R}\|f(\tau,\cdot)\|_{L^p(\R^n,\omega)}\int_{\R}e^{-|t-\tau|}\nu(t)\,dt\,d\tau \\
&\leq C\int_\R\|f(\tau,\cdot)\|_{L^p(\R^n,\omega)}M_{\R}\nu(\tau)\,d\tau \\
&\leq C\int_\R\|f(\tau,\cdot)\|_{L^p(\R^n,\omega)}\nu(\tau)\,d\tau=C\|f\|_{L^1(\R,\nu;L^p(\R^n,\omega))}.
\end{align*}
When $p=q=1$ we can estimate
\begin{align*}
\int_{\R^{n+1}}&|u(t,x)|\nu(t)\omega(x)\,dt\,dx \\
&\leq C\int_{\R^{n+1}}|f(\tau,y)|\int_\R e^{-|t-\tau|}\bigg[\int_{\R^n}\frac{e^{-|x-y|^2/(c|t-\tau|)}}{|t-\tau|^{n/2}}\omega(x)\,dx\bigg]
\nu(t)\,dt\,d\tau\,dy \\
&\leq C\int_{\R^{n+1}}|f(\tau,y)|M_{\R^n}\omega(y)\bigg[\int_\R e^{-|t-\tau|}\nu(t)\,dt\bigg]\,d\tau\,dy \\
&\leq C\int_{\R^{n+1}}|f(\tau,y)|M_{\R^n}\omega(y)M_{\R}\nu(\tau)\,d\tau\,dy \\
&\leq C\int_{\R^{n+1}}|f(\tau,y)|\omega(y)\nu(\tau)\,d\tau\,dy =C\|f\|_{L^1(\R,\nu;L^1(\R^n,\omega))}.
\end{align*}
This completes the proof of \eqref{eq:umixedweights}. 

Let $1<p<\infty$. Fix $\omega=\omega(x)\in A_p(\R^n)$ and let $E=F=L^p(\R^n,\omega)$.
Let $\nu=\nu(t)\in A_p(\R)$. It is easy to check that the tensor product weight
$w(t,x)=\nu(t)\omega(x)$ belongs to $PA_p(\R^{n+1})$.
Consider, as in the proof of Theorem \ref{thm:parabolic weights}, the singular integral operator
$$R_{ij}f(t,x)=\lim_{\varepsilon\rightarrow0^+}\iint_{\Omega_\varepsilon(t,x)}
K_{ij}(t,\tau,x-y)f(\tau,y)\,dy\,d\tau-f(t,x)I_{ij}(a)(t),$$
where $K_{ij}(t,\tau,x-y)$ is given as in \eqref{eq:kernelRiesztransformsx}.
Since $R_{ij}$ is bounded on $L^p(\mathbb{R}^{n+1},w)$, for any $w\in PA_p(\R^{n+1})$, we obviously have
that $R_{ij}$ is bounded on $L^p(\R;E)$.
For every $t,\tau\in\R$ such that $t\neq\tau$, define the vector-valued kernel
$\mathrm{K}_{ij}(t,\tau)\in\mathcal{L}(E,E)$ that acts on $\varphi\in E$ with compact support as
$$\big(\mathrm{K}_{ij}(t,\tau)\cdot\varphi\big)(x)=\int_{\R^n}K_{ij}(t,\tau,x-y)\varphi(y)\,dy.$$
Given a function $f(t,x)\in L^p(\R;E)$
we consider the vector-valued function
$t\longmapsto\mathrm{f}(t)\in E$,
given by $\mathrm{f}(t)=f(t,x)$, for a.e.~$t\in\R$.
We then define the vector-valued operator $\mathrm{R}_{ij}:L^p(\R;E)\to L^p(\R;E)$ as
$$\mathrm{R}_{ij}\mathrm{f}(t)=\int_{\R}\mathrm{K}_{ij}(t,\tau)\cdot\mathrm{f}(\tau)\,d\tau,
\quad\hbox{for}~t\notin\supp(\mathrm{f}).$$

The classical heat semigroup $e^{r\Delta}$ is bounded on $L^p(\R^n,\omega)$,
with constant independent of $r>0$. Whence, for any $\varphi\in E$, by \eqref{eq:tamanodeKij},
\begin{align*}
\|\mathrm{K}_{ij}(t,\tau)\cdot\varphi\|_E &= \bigg\|\int_{\R^n}
K_{ij}(t,\tau,y)\varphi(x-y)\,dy\bigg\|_{L^p(\R^n,\omega)} \\
&\leq C\chi_{t>\tau}\bigg\|\int_{\R^n}\frac{e^{-|y|^2/(c(t-\tau))}}{((t-\tau)^{1/2})^{n+2}}\varphi(x-y)\,dy\bigg\|_{L^p(\R^n,\omega)} \\
&= \frac{C\chi_{t>\tau}}{t-\tau}\|e^{c(t-\tau)\Delta}\varphi\|_{L^p(\R^n,\omega)} \\
&\leq \frac{C}{|t-\tau|}\|\varphi\|_{L^p(\R^n,\omega)}.
\end{align*}
Similarly, by using \eqref{eq:tamanioderivadasent},
\begin{align*}
\|\partial_t\mathrm{K}_{ij}(t,\tau)\cdot\varphi\|_E+\|\partial_\tau\mathrm{K}_{ij}(t,\tau)\cdot\varphi\|_E
&\leq C\chi_{t>\tau}\bigg\|\int_{\R^n}\frac{e^{-|y|^2/(c(t-\tau))}}{((t-\tau)^{1/2})^{n+4}}\varphi(x-y)\,dy\bigg\|_{L^p(\R^n,\omega)} \\
&= \frac{C\chi_{t>\tau}}{(t-\tau)^2}\|e^{c(t-\tau)\Delta}\varphi\|_{L^p(\R^n,\omega)} \\
&\leq \frac{C}{|t-\tau|^2}\|\varphi\|_{L^p(\R^n,\omega)}.
\end{align*}
These estimates show that $\mathrm{K}_{ij}(t,\tau)\in\mathcal{L}(E,E)$ is a vector-valued Calder\'on--Zygmund kernel
on the space of homogeneous type $(\R,|\cdot|,dt)$. Therefore Theorem \ref{thm:CZ} applies.
The same method also works for $f\longmapsto\partial_tu$. The proof of Theorem \ref{thm:LqSobolev} is complete.\qed

\subsection{Proof of Theorem \ref{thm:mixed-BMO}.}

In the proof of Theorem \ref{thm:LqSobolev} we showed that $R_{ij}$ are vector-valued
Calder\'on--Zygmund operators in the weighted mixed-norm spaces, and similarly for $f\longmapsto\partial_tu$.
The conclusion follows from Theorem \ref{thm:Segovia-Torrea} with $\X=\R$,
$E=F=L^p(\R^n,\omega)$.\qed

\section{Weighted mixed-norm Holder estimates: proof of Theorem \ref{thm:Holder}}

We denote by $C^{\alpha}(\R^n)$, $0<\alpha<1$, the space of continuous functions $\phi:\R^n\to\R$ such that
$$[\phi]_{C^\alpha(\R^n)}=\sup_{x\neq y}\frac{|\phi(x)-\phi(y)|}{|x-y|^\alpha}<\infty.$$
We make the identification $\phi_1=\phi_2$ in $C^{\alpha}(\R^n)$ if $\phi_1-\phi_2$ is constant. Within
this quotient space, $[\phi]_{C^\alpha(\R^n)}$ becomes a norm and $C^{\alpha}(\R^n)$
is a Banach space. A strongly measurable function $f=f(t,x)$ belongs to $L^q(\R,\nu;C^{\alpha}(\R^n))$, 
with $\nu\in A_q(\R)$ if $1\leq q<\infty$, and $\nu\equiv1$ if $q=\infty$, if
$$\|f\|_{L^q(\R,\nu;C^{\alpha}(\R^n))}=\|[f(t,\cdot)]_{C^\alpha(\R^n)}\|_{L^q(\R,\nu)}<\infty.$$
As before, we consider the singular integral operator
\begin{align*}
R_{ij}f(t,x)&=\lim_{\varepsilon\rightarrow0^+}\iint_{\Omega_\varepsilon(t,x)}
K_{ij}(t,\tau,x-y)f(\tau,y)\,dy\,d\tau-f(t,x)I_{ij}(a)(t)\\
&=T_{ij}f(t,x)-f(t,x)I_{ij}(a)(t),
\end{align*}
where $K_{ij}(t,\tau,x-y)$ is given as in \eqref{eq:kernelRiesztransformsx}.

\begin{lem}\label{Lemma6.1}
Let $f=f(t,x)\in L^\infty(\R;C^{\alpha}(\R^n))$ for some $0<\alpha<1$. Then
$$\|R_{ij}f\|_{L^\infty(\R;C^{\alpha}(\R^n))}\leq C\|f\|_{L^\infty(\R;C^{\alpha}(\R^n))}.$$
The singular integral operator $f\longmapsto\partial_tu$ satisfies the same estimate.
\end{lem}

\begin{proof}
The fact that $R_{ij}f$ is measurable as a $C^{\alpha}(\R^n)$-valued
function was proved by Krylov in \cite[p.~819]{Krylov}.
We only need to consider $T_{ij}f(t,x)$, because $I_{ij}(a)(t)\in L^\infty(\R)$.
Let $x_1,x_2\in\R^n$ and set $\rho=|x_1-x_2|^2$. Since
$$\int_{\R^n}K_{ij}(t,\tau,y)\,dy=0,$$
for every $t,\tau\in\R$, $i,j=1,\ldots,n$, we can write
\begin{align*}
&T_{ij}f(t,x_1)-T_{ij}f(t,x_2) \\
&= \int_0^\rho\int_{\R^n}\bar{K}_{ij}(t,\tau,y)\big(f(t-\tau,x_1-y)-f(t-\tau,x_1)\big)\,dy\,d\tau \\
&\quad+\int_0^\rho\int_{\R^n}\bar{K}_{ij}(t,\tau,y)\big(f(t-\tau,x_2-y)-f(t-\tau,x_2)\big)\,dy\,d\tau \\
&\quad+\int_\rho^\infty\int_{\R^n}\big(\bar{K}_{ij}(t,\tau,x_1-y)
-\bar{K}_{ij}(t,\tau,x_2-y)\big)\big(f(t-\tau,y)-f(t-\tau,x_1)\big)\,dy\,d\tau \\
&= I+II+III,
\end{align*}
where $\bar{K}_{ij}(t,\tau,y) = \partial_{y_iy_j}p(t,\tau,y)$ (compare with \eqref{eq:kernelRiesztransformsx}).
By applying \eqref{eq:tamanodeKij} we can estimate
\begin{align*}
|I|+|II| &\leq C\|[f(t,\cdot)]_{C^\alpha(\R^n)}\|_{L^\infty(\R)}\int_0^\rho e^{-\tau}
\int_{\R^n}\frac{e^{-|y|^2/(c\tau)}}{\tau^{n/2+1}}\frac{|y|^\alpha}{\tau^{\alpha/2}}
\tau^{\alpha/2}\,dy\,d\tau \\
&\leq C\|f\|_{L^\infty(\R;C^{\alpha}(\R^n))}\int_0^\rho\tau^{\alpha/2-1}\,d\tau \\
&= C\|f\|_{L^\infty(\R;C^{\alpha}(\R^n))}|x_1-x_2|^\alpha.
\end{align*}
For $III$, let $\displaystyle E(x) = \chi_{\tau>0}\, e^{-\tau}\, \frac{\exp\big(-\tfrac{1}{4}\langle B_{t,\tau}x,x\rangle\big)}{(4\pi)^{n/2} (\det B_{t,\tau})^{-1/2}}$ and $F(x)= [-\tfrac{1}{2}B_{ij} +\tfrac{1}{4}(Bx)_i(Bx)_j\Big].$  By using Lemma \ref{derivativex} we have
\begin{align*}
|\bar{K}_{ij}(t,\tau,&x_1-y)-\bar{K}_{ij}(t,\tau,x_2-y)| \\
&=  E(x_1-y) F(x_1-y) - E(x_2-y) F(x_2-y) \\
&= (E(x_1-y)  - E(x_2-y)) F(x_1-y) +E(x_2-y) ( F(x_1-y)- F(x_2-y)) \\
&= \mathcal{A}_1+ \mathcal{A}_2.
\end{align*}
Consider the function $\displaystyle s\longmapsto\chi_{\tau>0}\,e^{-\tau}\, \frac{\exp\big(-\tfrac{1}{4}s^2\big)}{(4\pi)^{n/2} (\det B_{t,\tau})^{-1/2}}$. By the mean value theorem  there exists  an intermediate real number $\eta $  between $|B_{t,\tau}^{1/2}(x_1-y)|$ and $|B_{t,\tau}^{1/2}(x_2-y)|$  (we  assume $|B_{t,\tau}^{1/2}(x_1-y)| \le \eta \le |B_{t,\tau}^{1/2}(x_2-y)|$), such that  
\begin{align*}
|\mathcal{A}_1| &\le \bigg|\frac{d}{ds}\bigg[\chi_{\tau>0}e^{-\tau}\frac{\exp\big(-\tfrac{1}{4}s^2\big)}{(4\pi)^{n/2} (\det B_{t,\tau})^{-1/2}}\bigg]\Big|_{s=\eta}\big||B_{t,\tau}^{1/2}(x_1-y)|-|B_{t,\tau}^{1/2}(x_2-y)|\big||F(x_1-y) |
\\ &\le  C   \chi_{\tau>0}e^{-\tau}\eta \frac{\exp\big(-\tfrac{1}{4}\eta^2\big)}{(\det B_{t,\tau})^{-1/2}\tau^{1/2}}
|x_1-x_2||F(x_1-y) | \\
&\le  C    \chi_{\tau>0}e^{-\tau} \frac{|x_2-y|}{\tau^{1/2}}\cdot \frac{\exp\big(-|x_1-y|^2/(c\tau)\big)}{\tau^{n/2}}\cdot\frac{ 
|x_1-x_2|}{\tau^{1/2}}\bigg[\frac{1}{\tau}+  \frac{|x_1-y|^2}{\tau^2}\bigg]\\ &\le
C    \chi_{\tau>0}e^{-\tau} \frac{(|x_2-x_1|+ |x_1-y|)}{\tau^{1/2}}\cdot \frac{\exp\big(-|x_1-y|^2/(c\tau)\big)}{\tau^{n/2}} 
\cdot\frac{|x_1-x_2|}{\tau^{1/2}} \cdot\frac{1}{\tau} \\ &=
C    \chi_{\tau>0}e^{-\tau} \bigg[\frac{|x_2-x_1|^2}{\tau^2} + \frac{|x_1-x_2||x_1-y| }{\tau^2}\bigg]\cdot
\frac{\exp\big(-|x_1-y|^2/(c\tau)\big)}{\tau^{n/2}}.
\end{align*}
On the other hand, it is routine to check that for $|B_{t,\tau}^{1/2}(x_1-y)|  \le |B_{t,\tau}^{1/2}(x_2-y)|$ one has
 \begin{align*}
 |\mathcal{A}_2| &\le C \chi_{\tau>0}e^{-\tau} \frac{|x_1-x_2|}{\tau^2} (|x_1-y|+|x_2-y|)
 \frac{\exp\big(-|x_1-y|^2/(c\tau)\big)}{\tau^{n/2}} \\
 &\le C \chi_{\tau>0}e^{-\tau} \bigg[\frac{|x_2-x_1|^2}{\tau^2} + \frac{|x_1-x_2||x_1-y| }{\tau^2}\bigg]
 \cdot \frac{\exp\big(-|x_1-y|^2/(c\tau)\big)}{\tau^{n/2}}.
 \end{align*}
Now we come back to the estimate for $III$. We have  
\begin{align*}
|III| &\leq C\|f\|_{L^\infty(\R;C^{\alpha}(\R^n))}\int_\rho^\infty e^{-\tau}\int_{\R^n}
\bigg[\frac{|x_2-x_1|^2}{\tau^2} + \frac{|x_1-x_2||x_1-y| }{\tau^2}\bigg]\\
&\qquad\qquad\qquad\qquad\qquad\qquad\qquad\qquad \times\frac{\exp\big(-|x_1-y|^2/(c\tau)\big)}{\tau^{n/2}}
|x_1-y|^\alpha\,dy\,d\tau \\
&\le  C\|f\|_{L^\infty(\R;C^{\alpha}(\R^n))}\int_\rho^\infty \int_{\R^n}
  \bigg[\frac{|x_1-x_2|^2}{\tau^3}+ \frac{|x_1-x_2|}{\tau^2}\bigg]
  \tau^{\alpha/2}\frac{\exp\big(-|x_1-y|^2/(c\tau)\big)}{\tau^{n/2}}
\,dy\,d\tau  \\ & 
= C\|f\|_{L^\infty(\R;C^{\alpha}(\R^n))}\int_\rho^\infty 
\bigg[\frac{|x_1-x_2|^2}{\tau^3}+ \frac{|x_1-x_2|}{\tau^2}\bigg]\tau^{\alpha/2}\,d\tau  \\
&= C\|f\|_{L^\infty(\R;C^{\alpha}(\R^n))} |x_1-x_2|^\alpha.
\end{align*}
Hence
$$|T_{ij}f(t,x_1)-T_{ij}f(t,x_2)|\leq C\|f\|_{L^\infty(\R;C^{\alpha}(\R^n))}|x_1-x_2|^\alpha,$$
for every $x_1,x_2\in\R^n$, uniformly in $t\in\R$.
\end{proof}

\begin{proof}[Proof of Theorem \ref{thm:Holder}]
Let $E=F=C^{\alpha}(\R^n)$ and, for every $t,\tau\in\R$ such that $t\neq\tau$, define the vector-valued kernel
$\mathrm{K}_{ij}(t,\tau)\in\mathcal{L}(E,E)$ that acts on $\varphi\in E$ with compact support as
$$\big(\mathrm{K}_{ij}(t,\tau)\cdot\varphi\big)(x)=\int_{\R^n}K_{ij}(t,\tau,x-y)\varphi(y)\,dy.$$
Given a function $f(t,x)\in L^\infty(\R;E)$
we consider the vector-valued function
$t\longmapsto\mathrm{f}(t)\in E$,
given by $\mathrm{f}(t)=f(t,x)$, for a.e.~$t\in\R$.
We then define the vector-valued operator $\mathrm{R}_{ij}:L^\infty(\R;E)\to L^\infty(\R;E)$ as
$$\mathrm{R}_{ij}\mathrm{f}(t)=\int_{\R}\mathrm{K}_{ij}(t,\tau)\cdot\mathrm{f}(\tau)\,d\tau,
\quad\hbox{for}~t\notin\supp(\mathrm{f}).$$

We first check that $\mathrm{K}_{ij}(t,\tau)\in\mathcal{L}(E,E)$ for every $t\neq\tau$
with the corresponding size estimate.
For any $\varphi\in E$, by \eqref{eq:tamanodeKij},
\begin{align*}
|\mathrm{K}_{ij}(t,\tau)\cdot\varphi(x_1)-\mathrm{K}_{ij}(t,\tau)\cdot\varphi(x_2)| &\leq \int_{\R^n}
|K_{ij}(t,\tau,y)||\varphi(x_1-y)-\varphi(x_2-y)|\,dy \\
&\leq C\chi_{t>\tau}[\varphi]_{C^\alpha(\R^n)}|x_1-x_2|^\alpha
\int_{\R^n}\frac{e^{-|y|^2/(c(t-\tau))}}{((t-\tau)^{1/2})^{n+2}})\,dy \\
&\leq \frac{C}{|t-\tau|}[\varphi]_{C^\alpha(\R^n)}|x_1-x_2|^\alpha,
\end{align*}
so that
$$\|\mathrm{K}_{ij}(t,\tau)\|\leq\frac{C}{|t-\tau|}.$$
Similarly, by using \eqref{eq:tamanioderivadasent}, 
\begin{align*}
[\partial_t\mathrm{K}_{ij}(t,\tau)\cdot\varphi]_{C^\alpha(\R^n)}
+[\partial_\tau\mathrm{K}_{ij}(t,\tau)\cdot\varphi]_{C^\alpha(\R^n)}
&\leq C\chi_{t>\tau}[\varphi]_{C^\alpha(\R^n)}\int_{\R^n}\frac{e^{-|y|^2/(c(t-\tau))}}{((t-\tau)^{1/2})^{n+4}}\,dy \\
&\leq \frac{C}{|t-\tau|^2}[\varphi]_{C^\alpha(\R^n)}.
\end{align*}
These estimates show that $\mathrm{K}_{ij}(t,\tau)\in\mathcal{L}(E,E)$ is a vector-valued Calder\'on--Zygmund kernel
on the space of homogeneous type $(\R,|\cdot|,dt)$. 
Therefore, Theorem \ref{thm:CZ} applies.
The case $f\longmapsto\partial_tu$ is similar and we leave the details to the reader.
\end{proof}

\section{The Cauchy problem: proof of Theorem \ref{thm:Cauchy}}\label{section:Cauchy}

\begin{proof}[Proof of Theorem \ref{thm:Cauchy}]
First, to prove that $v$ satisfies the integrability properties in mixed normed spaces with weights,
we can perform computations completely analogous to those done for the solution $u$
given by \eqref{u}. Details are therefore omitted.

Secondly, to obtain the estimates for the derivatives of $v$,
one would be tempted to apply again the Calder\'on--Zygmund theorem to the integral operators
in \eqref{equation3 x} and \eqref{equation3 t}.
Nevertheless, the set $\Sigma_\varepsilon$ appearing in the proof in Subsection \ref{subsection:cauchy}
does not correspond to the standard truncations for Calder\'on--Zygmund operators
as given by \eqref{truncation} and \eqref{conjuntitos}. 
Therefore we can not use the Calder\'on--Zygmund machinery as we did for $u$.
Instead, prove the results we will do a comparison argument with the global case of $\R^{n+1}$.
Consider the following difference
$$\bigg|\iint_{\Omega_\varepsilon }\partial_{y_iy_j}p(t,\tau,y) f(t-\tau,x-y)\chi_{0<\tau<t} \,dy\, d\tau- 
\int_{\varepsilon^2}^t \int_{\mathbb{R}^n}\partial_{y_iy_j}p(t,\tau,y) f(t-\tau,x-y) \,dy\, d\tau\bigg|.$$
We have 
$$\big|(\chi_{\Omega_\varepsilon}-\chi_{\Sigma_{\varepsilon^2}})\chi_{\tau<t}\partial_{y_iy_j}p(t,\tau,y)\big|
\le C\chi_{|y|>\varepsilon}\chi_{\tau < \varepsilon^2}e^{-\tau}\frac{e^{-|y|^2/(c\tau)}}{\tau^{n/2+1}}.$$
Since $\tau<\varepsilon^2<|y|^2$, we get $|y|^2/(c\tau)\geq C$ and therefore, for any $m>0$ there
exists $C_m>0$ such that
$$e^{-|y|^2/(c\tau)}\leq C_m\bigg(\frac{\tau}{|y|^2}\bigg)^{m/2}$$
and, in particular,
$$\frac{e^{-|y|^2/(c\tau)}}{\tau^{n/2+1}}\leq C_n\bigg(\frac{1}{|y|^2}\bigg)^{n/2+1}.$$
As a consequence,
\begin{align*}
\big|(\chi_{\Omega_\varepsilon}-\chi_{\Sigma_{\varepsilon^2}})\chi_{\tau<t}\partial_{y_iy_j}p(t,\tau,x-y)\big|
&\le C\chi_{|y|>\varepsilon}\chi_{\tau < \varepsilon^2}\bigg(\frac{\tau}{|y|^2}\bigg)^{m/2}\bigg(\frac{1}{|y|^2}\bigg)^{n/2+1} \\
&\leq C\chi_{|y|>\varepsilon}\chi_{\tau < \varepsilon^2}\frac{\varepsilon^{m}}{(|y|^2+\tau)^{m/2+n/2+1}}\\
&= C\chi_{|y|>\varepsilon}\chi_{\tau < \varepsilon^2}\frac{1}{\varepsilon^{n+2}(|y|^2/\varepsilon^2+\tau/\varepsilon^2)^{m/2+n/2+1}}\\
&=\frac{1}{\varepsilon^{n+2}}\Psi\bigg(\frac{|y|}{\varepsilon},\frac{\tau}{\varepsilon^2}\bigg),
\end{align*}
with
\begin{align*}
\Psi(|y|,\tau) &= C(|y|^2+\tau)^{-(m/2+n/2+1)}\chi_{|y| >1}\chi_{\tau<1} \\
&\leq C\chi_{|y|^2+\tau\leq1}(y,\tau)+C\chi_{|y|^2+\tau>1}(|y|^2+\tau)^{-(m/2+n/2+1)}=:\Phi(y,\tau).
\end{align*}
The function $\Phi(y,\tau)$ is decreasing, radially symmetric and integrable in $\R^{n+1}$.
Therefore
\begin{multline*}
\sup_{\varepsilon>0}\bigg|\iint_{\Omega_\varepsilon }\partial_{y_iy_j}p(t,\tau,y) f(t-\tau,x-y)\chi_{0<\tau<t} \,dy\, d\tau- 
\int_{\varepsilon^2}^t \int_{\mathbb{R}^n}\partial_{y_iy_j}p(t,\tau,y) f(t-\tau,x-y) \,dy\, d\tau\bigg|\\
\le \sup_{\varepsilon>0} \int_{\mathbb{R}^{n+1}} \frac1{\varepsilon^{n+2}}\Phi\bigg(\frac{y}{\varepsilon},
\frac{\tau}{\varepsilon^2}\bigg)|f(t-\tau,x-y)|\,dy\,d\tau,
\end{multline*}
and this last operator is bounded in $L^p(\R^{n+1}_+,w)$, $w\in PA_p(\R^{n+1}_+)$, $1<p<\infty$,
and from $L^1(\R^{n+1}_+,w)$ into weak-$L^1(\R^{n+1}_+,w)$, $w\in PA_1(\R^{n+1}_+)$.

On the other hand, to prove the boundedness in mixed norm Sobolev spaces with weights, we notice that
\begin{align*}
\big|(\chi_{\Omega_\varepsilon}-\chi_{\Sigma_{\varepsilon^2}})\chi_{\tau<t}\partial_{y_iy_j}p(t,\tau,y)\big|
&\le C\chi_{|y|>\varepsilon}\chi_{\tau < \varepsilon^2}e^{-\tau}\frac{e^{-|y|^2/(c\tau)}}{\tau^{n/2+1}} \\
&\leq C\frac{e^{-|y|^2/(2c\tau)}e^{-\varepsilon^2/(2c\tau)}}{\tau^{n/2}\tau} \\
&= C\frac{e^{-|y|^2/(2c\tau)}}{\tau^{n/2}}\cdot\frac{1}{\varepsilon^2}\phi(\tau/\varepsilon^2)
\end{align*}
where $\phi(\tau)=e^{-1/(2c\tau)}$. From this kernel estimate, and by manipulating with maximal
functions as we did for $u$, the $L^q(\R,\nu;L^p(\R^n,\omega))$ and weak-type estimates can be derived.

For the $L^q(\R,\nu;C^{\alpha}(\R^n))$ estimate, observe that, for any $x_1,x_2\in\R^n$,
\begin{align*}
\int_{\R^{n+1}_+}\frac{e^{-|y|^2/(c\tau)}}{\tau^{n/2}}&\cdot\frac{1}{\varepsilon^2}\phi(\tau/\varepsilon^2)
|f(t-\tau,x_1-y)-f(t-\tau,x_2-y)|\,dy\,d\tau\\
&\leq C|x_1-x_2|^\alpha\int_\R\frac{1}{\varepsilon^2}\phi(\tau/\varepsilon^2)[f(t-\tau,\cdot)]_{C^\alpha(\R^n)}
\bigg[\int_{\R^n}\frac{e^{-|y|^2/(c\tau)}}{\tau^{n/2}}\,dy\bigg]\,d\tau \\
&= C|x_1-x_2|^\alpha\int_\R\frac{1}{\varepsilon^2}\phi(\tau/\varepsilon^2)[f(t-\tau,\cdot)]_{C^\alpha(\R^n)}\,d\tau\\
&\leq C|x_1-x_2|^\alpha M_{\R}\big\{[f(\cdot,\cdot)]_{C^\alpha(\R^n)}\big\}(t).
\end{align*}
The boundedness properties of the Hardy--Littlewood maximal function allow to prove the desired estimates.

Now we remind that for good enough functions  the limit 
$$\lim_{\varepsilon \rightarrow 0}\bigg\{\iint_{\Omega_\varepsilon}-\iint_{\Sigma_{\varepsilon^2}}\bigg\}\,
\partial_{y_iy_j}p(t,\tau,y)f(t-\tau,x-y)\,d\tau\,dy$$
exists (see \eqref{equation2 x} and \eqref{equation3 x}). 
An application of the Banach principle of almost everywhere convergence gives the proof
of the a.e. convergence of this last limit.   On the other hand, by Theorems  \ref{thm:parabolic weights}, \ref{thm:LqSobolev}
and \ref{thm:Holder}, we have all the necessary convergence results for the limit 
$$\lim_{\varepsilon \rightarrow 0} \iint_{\Omega_\varepsilon}\,
\partial_{y_iy_j}p(t,\tau,y)f(t-\tau,x-y)\,d\tau\,dy.$$
Hence we get the last conclusion of Theorem \ref{thm:Cauchy} for $\partial_{ij}v$.

For $\partial_tv$ we can proceed similarly, details are left to the interested reader.
\end{proof}

\section*{Appendix: Proof of Theorem \ref{thm:Segovia-Torrea}}

In this section we fix a weight $w>0$ such that $w^{-1}\in A_1(\X)$. We first obtain a couple of
properties about $w^{-1}$ that will be needed in the proof of Theorem \ref{thm:Segovia-Torrea}.

Weights in $A_1(\X)$ satisfy a reverse H\"older inequality, see \cite[Theorem~1]{Calderon}.
Hence there exist constants $\varepsilon>0$ and $C>0$ such that, for any ball $B\subset\X$
and any $1\leq r\leq1+\varepsilon$,
\begin{equation}\label{eq:reverse Holder}
\begin{aligned}
\bigg(\frac{1}{\mu(B)}\int_B(w^{-1})^r\,d\mu\bigg)^{1/r}
&\leq \bigg(\frac{1}{\mu(B)}\int_B(w^{-1})^{1+\varepsilon_1}\,d\mu\bigg)^{1/(1+\varepsilon_1)} \\
&\leq \frac{C}{\mu(B)}\int_Bw^{-1}\,d\mu \leq C\inf_Bw^{-1}.
\end{aligned}
\end{equation}

From \eqref{eq:reverse Holder} we immediately have that, for any $1\leq r\leq 1+\varepsilon$,
$$\frac{1}{\mu(B)}\int_B(w^{-1})^r\,d\mu\leq C\inf_Bw^{-r},$$
which means that $w^{-r}\in A_1(\X)$. Therefore,
as $A_1(\X)\subset A_p(\X)$ for any $p\geq1$ (see \cite[Proposition~7.2(1)]{Duo}),
we conclude that
$$w^{-r}\in A_r(\X),\quad\hbox{for any}~1\leq r\leq 1+\varepsilon.$$
In the case when $1<r\leq1+\varepsilon$, this last condition reads,
\begin{equation}\label{zzz}
\bigg(\int_B w^{\frac{r}{r-1}}\,d\mu\bigg)\bigg(\int_B w^{-r}\,d\mu\bigg)^{\frac{1}{r-1}}\leq
C\mu(B)^\frac{r}{r-1}.
\end{equation}
Let $1/r+1/r'=1$. Then $\frac{r}{r-1}=r'$, $-r=\frac{r'}{1-r'}$, and $\frac{1}{r-1}=r'-1$, so \eqref{zzz} can be written as
$$\bigg(\int_B w^{r'}\,d\mu\bigg)\bigg(\int_B (w^{r'})^{\frac{1}{1-{r'}}}\,d\mu\bigg)^{r'-1}\leq C\mu(B)^{r'}.$$
We have just proved that
\begin{equation}\label{eq:Bloom}
w^{r'}\in A_{r'}(\X),\quad\hbox{for every}~1<r\leq1+\varepsilon,~1/r+1/r'=1.
\end{equation}

With \eqref{eq:reverse Holder} and \eqref{eq:Bloom} at hand
we can now present the proof of Theorem \ref{thm:Segovia-Torrea}.

\begin{proof}[Proof of Theorem \ref{thm:Segovia-Torrea}]
Let $f\in L^\infty_c(\X;E)$ such that $wf\in L^\infty(\X;E)$.
Fix any metric ball $B=B(\x_0,\delta)\subset\X$ with center at $\x_0\in\X$ and radius $\delta>0$.
We use the notation $cB=B(\x_0,c\delta)$, for $c>0$.
Fix $R>5\kappa$, where $\kappa\geq1$ is the quasi-triangle inequality constant for the quasi-metric $\rho$.
We can write
$$f=f\chi_{\kappa RB}+f\chi_{(\kappa RB)^c}=:f_1+f_2,$$
Then $Tf=Tf_1+Tf_2$.
If $\x\in B$ then $\x\notin\supp(f_2)$ and we can write
$$Tf_2(\x)=\int_{\X}K(\x,\y)f(\y)\,d\mu(\y),$$
see \eqref{CZ2}. Hence we can define
\begin{equation}\label{eq:cB}
c_B=\frac{1}{\mu(B)}\int_BTf_2(\z)\,d\mu(\z)=\frac{1}{\mu(B)}\int_B\int_{\X}K(\z,\y)f_2(\y)\,d\mu(\y)\,d\mu(\z)\in F.
\end{equation}
Now,
\begin{equation}\label{eq:principio}
\begin{aligned}
\frac{1}{\mu(B)}\int_B\|Tf-c_B\|_F\,d\mu &\leq \frac{1}{\mu(B)}\int_{B}\|Tf_1\|_F\,d\mu
+\frac{1}{\mu(B)}\int_{B}\|Tf_2-c_B\|_F\,d\mu \\
&=:I_1+I_2.
\end{aligned}
\end{equation}
We estimate $I_1$ and $I_2$ separately.

We begin with $I_1$. Pick any $1<r<1+\varepsilon$. We apply H\"older's inequality, \eqref{eq:reverse Holder},
and the fact that $T$ is bounded
from $L^{r'}(\X,v;E)$ into $L^{r'}(\X,v;F)$, for $v=w^{r'}\in A_{r'}(\X)$ to get
\begin{equation}\label{eq:f1}
\begin{aligned}
I_1 &= \frac{1}{\mu(B)}\int_{B}w^{-1}\|Tf_1\|_Fw\,d\mu \\
&\leq\bigg(\frac{1}{\mu(B)}\int_Bw^{-r}\,d\mu\bigg)^{1/r}\bigg(\frac{1}{\mu(B)}\int_B\|Tf_1\|_F^{r'}w^{r'}\,d\mu\bigg)^{1/r'} \\
&\leq C\inf_Bw^{-1}\bigg(\frac{1}{\mu(B)}\int_B\|wf_1\|_E^{r'}\,d\mu\bigg)^{1/r'} \\
&\leq C\inf_Bw^{-1}\|wf\|_{L^\infty(\X;E)}.
\end{aligned}
\end{equation}

For $I_2$, by using that $\x\notin\supp(f_2)$ and the definition of $c_B$ in \eqref{eq:cB},
\begin{align*}
I_2&=\frac{1}{\mu(B)}\int_B\bigg\|\frac{1}{\mu(B)}\int_B\int_{\X}
\big(K(\x,\y)-K(\z,\y)\big)f_2(\y)\,d\mu(\y)\,d\mu(\z)\bigg\|_Fd\mu(\x) \\
&\leq\frac{1}{\mu(B)^2}\int_B\int_B\int_{\X}\|K(\x,\y)-K(\z,\y)\|\|w(\y)f_2(\y)\|_E\,w(\y)^{-1}
\,d\mu(\y)\,d\mu(\z)\,d\mu(\x) \\
&\leq\frac{\|wf\|_{L^\infty(\X;E)}}{\mu(B)^2}\int_B\int_B
\bigg[\int_{(\kappa RB)^c}\|K(\x,\y)-K(\z,\y)\|w(\y)^{-1}\,d\mu(\y)\bigg]d\mu(\z)\,d\mu(\x).
\end{align*}

We claim that, for any $\x,\z\in B$,
\begin{equation}\label{eq:falta}
J:=\int_{(\kappa RB)^c}\|K(\x,\y)-K(\z,\y)\|w(\y)^{-1}\,d\mu(\y)\leq C\inf_{B}w^{-1}.
\end{equation}

With \eqref{eq:falta} the conclusion follows. Indeed, we have
$$I_2\leq C\inf_Bw^{-1}\|wf\|_{L^\infty(\X;E)}.$$
By plugging this last estimate and \eqref{eq:f1} into \eqref{eq:principio} we get that, for
almost every $\x\in\X$,
$$w(\x)M^\#_F(Tf)(\x)=w(\x)\sup_{B\ni\x}\frac{1}{\mu(B)}\int_B\|Tf-c_B\|_F\,d\mu
\leq C\|wf\|_{L^\infty(\X;E)},$$
where $C$ depends only on $T$, $w$ and the structure. Thus
$$\|wM^\#_F(Tf)\|_{L^\infty(\X)}\leq C\|wf\|_{L^\infty(\X;E)},$$
as desired.

For the proof of \eqref{eq:falta}, let us recall that $B=B(\x_0,\delta)$ and
let $A_k=((\kappa R)^k B)\setminus((\kappa R)^{k-1}B)$, for $k\geq2$.
For any $\x,\z\in B$ we have $\rho(\x,\z)<2\kappa\delta$. If $\y\in(\kappa RB)^c$ then
$$\rho(\y,\x_0)\leq\kappa\rho(\y,\z)+\kappa\rho(\z,\x_0)\leq \kappa\rho(\y,\z)+\kappa\delta
\leq \kappa\rho(\y,\z)+\rho(\y,\x_0)/R,$$
which gives
$$\rho(\y,\z)\geq \frac{R-1}{\kappa R}\rho(\y,\x_0).$$
Moreover, since $\kappa\geq1$,
$$\kappa R\delta\leq\rho(\x_0,\y)\leq \kappa\rho(\x_0,\z)+\kappa\rho(\z,\y)\leq
\kappa^2\delta+\kappa\rho(\y,\z),$$
so, as $R>5\kappa$, this implies that $4\kappa\delta\leq(R-\kappa)\delta\leq\rho(\y,\z)$. Hence
$$\rho(\y,\z)\geq4\kappa\delta>2\rho(\x,\z).$$
Also, if $\y\in A_k$ then $\rho(\y,\x_0)\geq (\kappa R)^{k-1}\delta$. 
Finally, observe that
$$B(\x_0,\tfrac{R-1}{4\kappa^2R}\rho(\y,\x_0))\subset B(\z,\rho(\y,\z)).$$
In fact, for $\w\in B(\x_0,\frac{R-1}{4\kappa^2R}\rho(\y,\x_0))$,
\begin{align*}
\rho(\w,\z) &\leq \kappa\rho(\w,\x_0)+\kappa\rho(\x_0,\z) \leq \frac{R-1}{4\kappa R}\rho(\y,\x_0)+\kappa\delta \\
&\leq \frac{1}{4}\rho(\y,\z)+\frac{1}{4}\rho(\y,\z)\leq\rho(\y,\z),
\end{align*}
and the conclusion follows. Then we can use 
the smoothness condition given in Definition \ref{CZ}(II.2) and H\"older's inequality
for some $1<r<1+\varepsilon$ to get
\begin{align*}
J &= \sum_{k=2}^\infty\int_{A_k}\|K(\x,\y)-K(\z,\y)\|w(\y)^{-1}\,d\mu(\y) \\
&\leq \sum_{k=2}^\infty\bigg(\int_{A_k}\frac{\rho(\x,\z)^{r'}}{\rho(\y,\z)^{r'}\mu(B(\z,\rho(\y,\z)))^{r'}}\,d\mu(\y)\bigg)^{1/r'}
\bigg(\int_{(\kappa R)^{k}B}w^{-r}\,d\mu\bigg)^{1/r} \\
&\leq C\Bigg[\sum_{k=2}^\infty(2\kappa\delta)\frac{\kappa R}{R-1}
\bigg(\int_{A_k}\frac{1}{\rho(\y,\x_0)^{r'}\mu(B(\x_0,\frac{R-1}{4\kappa^2R}\rho(\y,\x_0)))^{r'}}
\,d\mu(\y)\bigg)^{1/r'}\\
&\qquad\qquad\times\mu((\kappa R)^kB)^{1/r}\inf_{(\kappa R)^kB}w^{-1}\Bigg] \\
&\leq C_\kappa\inf_{B}w^{-1}\sum_{k=2}^\infty\delta\frac{1}{(\kappa R)^{k-1}\delta}
\frac{\mu((\kappa R)^kB)^{1/r'}}{\mu(B(\x_0,\frac{R-1}{4\kappa^2R}(\kappa R)^{k-1}\delta))}
\mu((\kappa R)^kB)^{1/r} \\
&\leq C_{\kappa}\inf_{B}w^{-1}\sum_{k=2}^\infty\frac{1}{(\kappa R)^{k-1}}
\frac{\mu((\kappa R)^kB)}{\mu(B(\x_0,(\kappa R)^{k-2}\delta))} \\
&= C_{\kappa}\inf_{B}w^{-1}\sum_{k=2}^\infty\frac{1}{(\kappa R)^{k-1}}
\frac{\mu((\kappa R)^kB)}{\mu((\kappa R)^{k-2}B)}\\
&\leq C\inf_Bw^{-1}\sum_{k=2}^\infty\frac{1}{(\kappa R)^{k-1}}=C\inf_Bw,
\end{align*}
where in the last inequality we used that
$$\mu((\kappa R)^{k}B)=\mu((\kappa R)^2(\kappa R)^{k-2}B)\leq C_d((\kappa R)^2)^\gamma\mu((\kappa R)^{k-2}B),$$
where $\gamma>0$ depends only on the doubling constant $C_d$, see Calder\'on \cite[Lemma~1]{Calderon}.
\end{proof}




\begin{thebibliography}{10}

\bibitem{Calderon} A.~P.~Calder\'on,
{Inequalities for the maximal function relative to a metric},
\textit{Studia Math.}
\textbf{57} (1976), 297--306.

\bibitem{CalderonBull} A.~P.~Calder\'on,
{Singular integrals},
\textit{Bull. Amer. Math. Soc.}
\textbf{72} (1966), 427--465.

\bibitem{Calderon-Zygmund} A.~P.~Calder\'on and A.~Zygmund,
{On the existence of certain singular integrals},
\textit{Acta Math.}
\textbf{88} (1952), 85--139.

\bibitem{DKadv} H.~Dong and D.~Kim,
{Elliptic and parabolic equations with measurable coefficients in weighted Sobolev spaces},
\textit{Adv. Math.}
\textbf{274} (2015), 681--735.

\bibitem{DKtran} H.~Dong and D.~Kim,
{On $L_p$-estimates for elliptic and parabolic equations with $A_p$ weights},
\textit{Trans. Amer. Math. Soc.}
\textbf{370} (2018), 5081--5130.

\bibitem{Duo} J.~Duoandikoetxea,
\textit{Fourier Analysis},
Graduate Studies in Mathematics \textbf{29},
American Mathematical Society,
Providence, RI, 2001.

\bibitem{Fabes} E.~B.~Fabes,
{Singular integrals and partial differential equations of parabolic type},
\textit{Studia Math.}
\textbf{28} (1966), 81--131.

\bibitem{Fabes-Sadosky} E.~B.~Fabes and C.~Sadosky,
{Poitnwise convergence for parabolic singular integrals},
\textit{Studia Math.}
\textbf{26} (1966), 225--232.

\bibitem{Jones} B.~F.~Jones, Jr.
{A class of singular integrals},
\textit{Amer. J. Math.}
\textbf{86} (1964), 441--462.

\bibitem{Krylov book} N.~V.~Krylov,
\textit{Lectures on Elliptic and Parabolic Equations in H\"older Spaces},
Graduate Studies in Mathematics \textbf{12},
American Mathematical Society,
Providence, R.I., 1996.

\bibitem{Krylov0} N.~V.~Krylov,
{The Calder\'on-Zygmund theorem and its applications to parabolic equations},
(Russian) \textit{Algebra i Analiz}
\textbf{13} (2001), 1--25;
translation in \textit{St. Petersburg Math. J.}
\textbf{13} (2002), 509--526.

\bibitem{Krylov} N.~V.~Krylov,
{The Calder\'on-Zygmund theorem and parabolic equations in $L_p(\R,C^{2+\alpha})$-spaces},
\textit{Ann. Scuola Norm. Sup. Pisa Cl. Sci (5)}
\textbf{I} (2002), 799--820.

\bibitem{MST} R.~A.~Mac\'ias, C.~Segovia and J.~L.~Torrea,
{Singular integral operators with non-necessarily bounded kernels on spaces of homogeneous type},
\textit{Adv. Math.}
\textbf{93} (1992), 25--60.

\bibitem{PST} L.~Ping, P.~R.~Stinga and J.~L.~Torrea,
{On weighted mixed-norm Sobolev estimates for some basic parabolic equations},
\textit{Comm. Pure Appl. Anal.}
\textbf{16} (2017), 855--882.

\bibitem{RRT} J.~L.~Rubio de Francia, F.~J.~Ruiz and J.~L.~Torrea,
Calder\'on-Zygmund theory for operator-valued kernels,
\textit{Adv. in Math.}
\textbf{62} (1986), 7--48.

\bibitem{Francisco} F.~J.~Ruiz and J.~L.~Torrea,
{Parabolic differential equations and vector-valued Fourier analysis},
\textit{Colloq. Math.}
\textbf{58} (1989), 61--75.

\bibitem{RT} F.~J.~Ruiz and J.~L.~Torrea,
{Vector-valued Calder\'on-Zygmund theory and Carleson measures on spaces of homogeneous nature},
\textit{Studia Math.}
\textbf{88} (1988), 221--243.

\bibitem{Segovia-Torrea} C.~Segovia and J.~L.~Torrea,
{Extrapolation for pairs of related weights},
in: \textit{Analysis and Partial Differential Equations}, 331--345,
Lecture Notes in Pure and Appl. Math. \textbf{122}, Dekker, New York, 1990.

\end{thebibliography}
\end{document}